\documentclass[11pt]{amsart}
\usepackage{amsmath, amssymb, amsbsy, amsfonts, amsthm, latexsym, amsopn, amstext, amsxtra, euscript, amscd, color, mathrsfs}
\usepackage[normalem]{ulem}
\usepackage{soul}

\PassOptionsToPackage{hyphens}{url}\usepackage{hyperref}
\makeatletter
\@namedef{subjclassname@2020}{%
  \textup{2020} Mathematics Subject Classification}
\makeatother
 
 \usepackage[capbesideposition=outside,capbesidesep=quad]{floatrow}

\restylefloat{table}
     \usepackage{booktabs}       
\usepackage{multirow,caption}
\usepackage{pifont}
\usepackage{amscd}
\usepackage{color,enumerate}
\newcommand{\RNum}[1]{\lowercase\expandafter{\romannumeral #1\relax}}

\usepackage[colorinlistoftodos,prependcaption,textsize=tiny]{todonotes}

\newtheorem{thm}{Theorem}[section]
\newtheorem{lem}[thm]{Lemma}

\newtheorem{prop}[thm]{Proposition}
\newtheorem{exmp}[thm]{Example}

\newtheorem{rmk}[thm]{Remark}

\newtheorem{thm-con}[thm]{Theorem-Conjecture}
\numberwithin{equation}{section}

\theoremstyle{definition}
\newtheorem{defn}[thm]{Definition}

\newcommand{\F}{\mathbb F}

\def\Tr{{\rm Tr}}

\def\Trq{{\rm Tr}_q^{q^2}}

\begin{document}
\title[New classes of trace-form permutation polynomials]{New classes of trace-form permutation polynomials and their compositional inverses}

\author[K. Garg]{Kirpa Garg}
\address{Department of Computer Science, University of Rouen Normandy, Rouen 76000, France}
\email{kirpa.garg@gmail.com}
\maketitle

\begin{abstract}
We construct permutation polynomials of the form $x+\gamma\,\Trq(h(x))$ over the
finite field $\F_{q^2}$. More precisely, we present two families of permutation
polynomials whose coefficients range over all elements of the field rather than being
restricted to a proper subfield. We also compute the compositional inverses of
both families. Our techniques involve the evaluation of certain
Kloosterman sums together with the analysis of some equations over finite
fields, and we believe these can be of independent interest.
\end{abstract}

\noindent\textbf{Keywords:} Permutation polynomials, finite fields, compositional inverse, character sums, trace function.

\noindent\textbf{MSC 2020:}  11T06, 12E20.

\section{Introduction}
Let $q=p^m$, where $p$ is a prime and $m \geq 1$ is an integer, and let $\F_q$ be the finite field with $q$ elements. As usual, $\F_q^*$ stands for the multiplicative group $\F_q \setminus \{0\}$ and $\F_q[x]$ for the polynomial ring over $\F_q$. Recall that each function from $\F_q$ to itself corresponds to exactly one polynomial in $\F_q[x]$ of degree less than $q$. A polynomial $f(x) \in \F_q[x]$ is said to be a permutation polynomial (PP) of $\F_q$ whenever the map $a \mapsto f(a)$ acts bijectively on $\F_q$. Owing to their extensive applications in coding theory~\cite{Ding, La}, cryptography~\cite{Bertoni, Dwork, SC}, combinatorial designs~\cite{DY}, and various other areas of mathematics and engineering, permutation polynomials over finite fields have attracted considerable attention over the past several decades. 

Permutation polynomials with a simple algebraic form are of particular interest. In particular, polynomials of the shape
$G(x)+\gamma\,\Tr_q^{q^n}\!\big(H(x)\big), \gamma\in\F_{q^n}^{*},$
are significant because of their close ties with several central objects in cryptography and coding theory: the passage from $G(x)$ to $G(x)+\gamma\Tr_q^{q^n}(H(x))$ is an instance of the switching method of Edel and Pott~\cite{EP}, originally devised in the study of APN functions, while the trace term links such polynomials to Boolean functions and their linear structures. This form was subsequently revisited by Charpin and Kyureghyan~\cite{CK}, who
derived six families of PPs of $\F_{2^n}$ by exploiting linear structures of Boolean functions.

In 2016, Kyureghyan and Zieve~\cite{KZ} thoroughly analysed the case
$G(x)=x$ and $H(x)=x^{k}$ with $\gamma\in\F_{q^n}^{*}$, $q$ odd, $n>1$ and $q^n<5000$, and organised all
but five of the resulting examples into nine infinite families. Motivated by
this, Ma and Ge~\cite{Ma} extended two of the exceptional cases to infinite
families. In even characteristic, Li et al.~\cite{Li} carried out the analogous
programme for $cx+\Tr_q^{q^n}(x^k)$ with $q=2^m$, $c\in\F_{q^n}^{*}$, $m>1$ and
$mn<14$, producing fifteen new families, again with four sporadic examples left
unexplained. In 2019, Zha, Hu and Zhang~\cite{Zha} gave new constructions
accounting for the remaining exceptional cases of~\cite{KZ} and~\cite{Li}.
Furthermore, Yuan~\cite{Yuan} introduced a novel algebraic framework for
permutations of $\F_{q^n}$ and used it to produce new families of PPs of the
form $G(x)+\gamma\,\Tr_q^{q^n}\!\big(H(x)\big)$; notably, his construction resolved an open problem posed by
Charpin and Kyureghyan~\cite{CK}, which called for a characterization of PPs of
this type in which $G(x)$ is neither a permutation nor a linearised polynomial.
 
Recently, Jiang et al.\ \cite{Jiang} characterized when
$f(x)=x+\gamma\Tr_q^{q^2}(h(x))$ permutes $\F_{q^2}$, where $q=2^m$,
$\gamma\in\F_{q^2}^{*}$ and $h(x)=c_1x+c_2x^2+c_3x^3+x^{q+2}$ or
$h(x)=c_1x+c_2x^2+c_3x^3+c_3x^{q+2}+x^{2q-1}$, with $c_i\in\F_2$.
In every construction of this shape known to us, the coefficients of $h$ are
confined to a proper subfield of $\F_{q^2}$. While this is natural for the sake
of computation, it turns out to be a genuine restriction rather than a mere
convenience. Indeed, the conditions determining whether $f$ permutes $\F_{q^2}$ involve
$\Trq(c_1)$ and $\Trq(c_2)$, both of which vanish for all $c_1,c_2\in\F_q$ in
even characteristic. In particular, the branch of
Theorem~\ref{Result1} below that requires $\Tr_q^{q^2}(c_2)\neq0$ is
unreachable from the base field, while in the remaining branches the
coefficients enter only through $\gamma$ or through a single scalar
combination.

In this paper we give necessary and sufficient conditions on
$c_1,c_2,c_3,\gamma\in\F_{q^2}$ for the resulting polynomial to permute
$\F_{q^2}$, and we compute their corresponding compositional inverses. Our results include families of permutation polynomials with $c_2\in\F_{q^2}\setminus\F_q$, as
well as families with $\gamma\in\F_{q^2}\setminus\F_q$, for which a characterization is not yet known.
Specialising the coefficients recovers several previously known results~\cite{HKK, Jiang, SKP}
within a single framework. 

The remainder of this paper is organised as follows. In Section~\ref{S2}
we recall the definitions and lemmas used later on. In Section~\ref{S3}
we give a complete characterization of the two families of permutation
polynomials described above, and in Section~\ref{S4} we determine their
compositional inverses. In Section~\ref{S5} we investigate
quasi-multiplicative equivalence between the permutation polynomials
obtained here and the known classes in the literature, thereby confirming
that the new families are genuinely inequivalent to those already known.
Section~\ref{S6} concludes the paper.

\section{Preliminaries}\label{S2}
Throughout the paper, we denote $q=2^m$. Throughout the paper, we denote by $\mbox{Tr}_{2^m}^{2^n}$ the (relative) trace function from $\F_{2^n} \rightarrow \F_{2^m}$, i.e., $ \mbox{Tr}_{2^m}^{2^n}(x)=\sum^{\frac{n-m}{m}}_{i=0}x^{2^{mi}},$ where $m$ and $n$ are positive integers with $m \mid n$. We now recall some results that we will use in the subsequent sections.

\begin{defn}{(Kloosterman Sum)}
Let $\chi$ be a non-trivial additive character of $\F_q$ and let $a,b \in \F_q$. Then the sum 
$$
K(\chi;a,b)=\sum_{x \in \F_q^*}\chi (ax+bx^{-1})
$$
is called a Kloosterman sum.
\end{defn}

In the following, we recall characterization of factors of  a cubic polynomial over finite field of even characteristic. 
\begin{lem}\label{L1}
Let $a, b \in \F_{2^n}$ , $b \neq 0$ and $f(x) = x^3 + ax + b, g(y) = y^2 + by + a^3$. Let $y_1$, $y_2$ be two roots of $g(y)$. The factorizations of $f(x)$ are characterized as follows:
\begin{enumerate}
\item $f = (1, 1, 1)$ if and only if $\Tr_2^{2^n} \left( \frac{a^3}{b^2} \right) = \Tr_{2}^{2^n}(1)$, and $y_1, y_2$ are cubes in $\F_{2^n}$ (resp. $\F_{2^{2n}}$) if $n$ is even (resp. odd);
\item $f = (1, 2)$ if and only if $\Tr_2^{2^n} \left( \frac{a^3}{b^2} \right) \neq \Tr_2^{2^n}(1)$;
\item $f = (3)$ if and only if $\Tr_2^{2^n} \left( \frac{a^3}{b^2} \right) = \Tr_2^{2^n}(1)$, and $y_1, y_2$ are not cubes in $\F_{2^n}$ (resp. $\F_{2^{2n}}$) if $n$ is even (resp. odd).
\end{enumerate}
\end{lem}

\section{Two families of permutation polynomials}\label{S3}

The following lemma will be used to determine necessary and sufficient conditions for our first family of polynomials to be a permutation.

\begin{lem}\label{im}
 Let \(q=2^m\), and let \(\alpha \in \mathbb{F}_{q^2}\setminus \mathbb{F}_q\), then $\phi:\mathbb{F}_{q^2}\to\mathbb{F}_q$ defined as
 $$\phi(x)=\Trq\!\big(\alpha(x^2+x)\big)$$
 is surjective on $\F_q$.
\end{lem}
\begin{proof}
Since $\phi$ is $\F_2$-linear, by the rank--nullity theorem it suffices to show that $\dim_{\F_2}\ker\phi=m$. Write $\phi=\Lambda\circ \eta$, where $\eta(x)=x^2+x$ and
$\Lambda(z)=\operatorname{Tr}_{q^2/q}(\alpha z)$. The map $\Lambda$ is
$\mathbb{F}_q$-linear, while $\eta$ is $\mathbb{F}_2$-linear with
$\ker \eta=\{0,1\}=\mathbb{F}_2$; hence $\eta$ is two-to-one onto $\operatorname{Im}(\eta)$. From~\cite{LN}[Theorem 2.25], we can write $$\operatorname{Im}(\eta)=\{x\in\mathbb{F}_{q^2}:\operatorname{Tr}_{2}^{q^2}(x)=0\}=\ker(\Tr_{2}^{q^2}(.)).$$
Since $\ker\phi=\eta^{-1}(\ker\Lambda \cap  \operatorname{Im}(\eta))=\{z \in \ker\Lambda \mid \Tr_{2}^{q^2}(z)=0\}$ and $\eta$ is 2-to-1 onto $\operatorname{Im}(\eta)$,
\[
  |\ker\phi|=2\,\big|\ker\Lambda\cap \operatorname{Im}(\eta)\big|,
\]
where $\ker(\Lambda)= \{z \in \F_{q^2} \mid (\alpha z)^q = \alpha z\}=\alpha^{-1} \F_q.$

For $\lambda\in\F_q$ we have
$\Tr_2^{q^2}(\alpha^{-1}\lambda)=\Tr_2^{q}\!\big(\lambda\,\Trq(\alpha^{-1})\big)$,
and $\Trq(\alpha^{-1})\neq 0$ since $\alpha\notin\F_q$. Hence the
restriction of $\Tr_2^{q^2}$ to $\ker\Lambda$ is a nonzero $\F_2$-linear
functional, so its kernel $\ker\Lambda\cap\operatorname{Im}(\eta)$ has
exactly $q/2$ elements. Therefore $|\ker\phi|=q$, that is,
$\dim_{\F_2}\ker\phi=m$.
\end{proof}

We first considered the case when $\gamma \not\in \F_q$, the case 
$\gamma \in \F_q$ is treated in the proof of Theorem~\ref{Result1}.

\begin{lem}\label{Re1}
Let $q=2^m$ with $m>1$, let $c_1,c_2,c_3\in\F_{q^2}$ and let
$\gamma\in\F_{q^2}\setminus\F_q$. Put
\[
  A:=\Trq\!\bigl(c_3\gamma^3+\gamma^{q+2}\bigr).
\]
Then
\[
  f_1(x)=x+\gamma\Trq\!\bigl(c_1x+c_2x^2+c_3x^3+x^{q+2}\bigr)
\]
is a permutation polynomial of $\F_{q^2}$ if and only if $m$ is odd,
\[
  A\neq 0,\qquad
  \Trq(\gamma^2)+(c_3\gamma+\gamma^{q})^{q+1}=0,\quad
  \mbox{~and~} \quad A\bigl(1+\Trq(c_1\gamma)\bigr)+\Trq(c_2^2\gamma^4)=0 .
\]
\end{lem}

\begin{proof}
To show $f_1(x)$ is a permutation polynomial over $\F_{q^2}$, it is sufficient to
prove that for any $\alpha\in\F_{q^2}$, the equation
$$x+\gamma \Tr_{q}^{q^2} (c_1 x+ c_2 x^2+c_3 x^3+x^{q+2})=\alpha$$
has a unique solution in $\F_{q^2}$. Substitute $x=\alpha+\gamma u$, $u\in\F_q$,
in $f_1(x)=\alpha$; we get the following cubic in the variable $u$
\begin{equation}\label{e1}
  A\,u^3+B(\alpha)\,u^2+C(\alpha)\,u+D(\alpha)=0,
\end{equation}
where, for $y\in\F_{q^2}$,
\begin{equation}\label{ABCD}
\begin{split}
  &A=\Trq\!\bigl(c_3\gamma^3+\gamma^{q+2}\bigr),\qquad
   B(y)=\Trq\!\bigl((c_2+c_3y+y^{q})\gamma^{2}\bigr),\\[2pt]
  &C(y)=1+\Trq\!\bigl(c_1\gamma+(c_3\gamma+\gamma^{q})y^{2}\bigr),\qquad
   D(y)=\Trq\!\bigl(c_1y+c_2y^2+c_3y^3+y^{q+2}\bigr).
\end{split}
\end{equation}
Clearly $f_1(x)=\alpha$ has a unique solution $x$ in $\F_{q^2}$ for every
$\alpha\in\F_{q^2}$ if and only if Equation~\eqref{e1} has a unique solution $u$
in $\F_q$ for every $\alpha\in\F_{q^2}$.

Dividing~\eqref{e1} by $A$ and substituting $u=z+B(\alpha)/A$ removes the
quadratic term and yields
\[
  z^3+\left(\frac{B(\alpha)^2}{A^2}+\frac{C(\alpha)}{A}\right)z
  +\frac{B(\alpha)C(\alpha)}{A^2}+\frac{D(\alpha)}{A}=0 .
\]
For every $\alpha\in\F_{q^2}$, the coefficient of $z$ vanishes, since
\[
  B(\alpha)^{2}+A\,C(\alpha)=\Trq\!\big(U\alpha^{2q}\big)+V,
\]
where
\[
  U=\gamma^{2}\Big(\Trq(\gamma^{2})+(c_3\gamma+\gamma^{q})^{q+1}\Big),
  \qquad
  V=A\big(1+\Trq(c_1\gamma)\big)+\Trq(c_2^{2}\gamma^{4}),
\]
and the hypotheses give $U=V=0$. Therefore, the cubic reduces to
\[
  z^3=\frac{B(\alpha)C(\alpha)+D(\alpha)A}{A^2}.
\]
Since $m$ is odd we have $\gcd(3,q-1)=1$, so $z\mapsto z^3$ permutes $\F_q$ and
this equation has exactly one solution $z\in\F_q$. Hence~\eqref{e1} has exactly
one solution $u\in\F_q$ for every $\alpha\in\F_{q^2}$, and therefore $f_1$
permutes $\F_{q^2}$.

Conversely, for $(x,u)\in\F_{q^2}\times\F_{q^2}^{*}$,
\[
  f_1(x+u)+f_1(x)=0
  \iff
  f_1(u)+\gamma\Trq\!\big(c_3x^2u+c_3xu^2+u^qx^2+u^2x^q\big)=0.
\]
Raising this equation to the $q$-th power and adding it to the original
yields $\gamma u^q=\gamma^q u$, so $u=\lambda\gamma$ for some
$\lambda\in\F_q^{*}$. Substituting back gives
$A\lambda^3+B(x)\lambda^2+C(x)\lambda=0$, and dividing by $\lambda\neq0$ we
obtain the quadratic
\begin{equation}\label{case61}
  A\lambda^2+B(x)\,\lambda+C(x)=0,
\end{equation}
with $A$, $B$, $C$ as in~\eqref{ABCD}.

First suppose $A=0$. For $S\in\F_{q^2}[x]$ write
$Z_S=\{x\in\F_{q^2}:S(x)=0\}$. Since
\[
  B(x)=0\iff\Trq\!\big((c_3\gamma^2+\gamma^{2q})x\big)=\Trq(c_2\gamma^2),
\]
\[
  C(x)=0\iff\Trq\!\big((c_3\gamma+\gamma^{q})x^{2}\big)=1+\Trq(c_1\gamma),
\]
and $x\mapsto x^2$ is a bijection of $\F_{q^2}$, each of $B$ and $C$, unless
identically zero, vanishes on exactly $q$ elements of $\F_{q^2}$. Hence, if
neither vanishes identically, then as $m>1$ we have $q\geq4$ and there are at
least $q(q-2)$ elements $x\in\F_{q^2}$ with $B(x)C(x)\neq0$; for any such $x$,
$\lambda=C(x)/B(x)\in\F_q^{*}$ satisfies~\eqref{case61}, and $f_1$ is not a
permutation of $\F_{q^2}$.

Now $B\equiv0$ if and only if $c_3\gamma^2+\gamma^{2q}=0$ and
$\Trq(c_2\gamma^2)=0$; note that $c_3\gamma^2+\gamma^{2q}=0$ already
forces $A=0$. In this case \eqref{case61} reduces to
\[
  \Trq\!\big(\gamma^{q-1}(\gamma+\gamma^q)\,x^2\big)=1+\Trq(c_1\gamma).
\]
As $x\mapsto x^2$ is a bijection of $\F_{q^2}$ and
$z\mapsto\Trq(\beta z)$ is onto $\F_q$ for
$\beta=\gamma^{q-1}(\gamma+\gamma^q)\neq 0$ (recall
$\gamma\notin\F_q$), this equation has a solution $x\in\F_{q^2}$, which for all
$\lambda\in\F_{q}^*$ gives us that $f_1(x+\lambda\gamma)=f_1(x)$. If instead
$C\equiv 0$, then $c_3=\gamma^{q-1}$, which again forces $A=0$, and
\eqref{case61} becomes $B(x)=0$, i.e.\
$\Trq\!\big((c_3\gamma^2+\gamma^{2q})x\big)=\Trq(c_2\gamma^2)$; since
$c_3\gamma^2+\gamma^{2q}=\gamma^{q}(\gamma+\gamma^{q})\neq0$, the same
surjectivity argument applies. Finally, $B$ and $C$ cannot both vanish
identically: $B\equiv 0$ forces $c_3=\gamma^{2(q-1)}$ while
$C\equiv 0$ forces $c_3=\gamma^{q-1}$, whence $\gamma^{q-1}=1$,
contradicting $\gamma\notin\F_q$.

Now let $A\neq 0$. Since $A=\Trq\!\big(\gamma^2(c_3\gamma+\gamma^q)\big)$, we
have $c_3\gamma+\gamma^q\neq0$, and also $c_3\gamma^2+\gamma^{2q}\neq0$ (its
vanishing forces $A=0$, as shown above). By the descriptions of $Z_B$ and $Z_C$
above, the left-hand sides are surjective maps onto $\F_q$ with fibres of size
$q$, so $|Z_B|=|Z_C|=q$. Recall that,
\[
  AC(x)+B(x)^2=\Trq\!\big(Ux^{2q}\big)+V,
  \qquad\text{where}
\]
\[
  U=\gamma^2\big(\Trq(\gamma^2)+(c_3\gamma+\gamma^q)^{q+1}\big),
  \qquad
  V=A\big(1+\Trq(c_1\gamma)\big)+\Trq(c_2^2\gamma^4)\in\F_q.
\]

Claim: $Z_B=Z_C$ if and only if $U=V=0$, if and only if
$AC\equiv B^2$.

Suppose first $U=V=0$. The displayed identity gives
$AC(x)=B(x)^2$ for all $x$, and since $A\neq0$,
$C(x)=0\iff B(x)=0$, i.e.\ $Z_B=Z_C$. Conversely, if
$AC\equiv B^2$, then $\Trq\!\big(Ux^{2q}\big)+V=0$ for all
$x\in\F_{q^2}$; taking $x=0$ gives $V=0$, and then, as
$x\mapsto x^{2q}$ is a bijection and the trace form is nondegenerate,
$U=0$.

It remains to show that $Z_B=Z_C$ implies $U=V=0$. Substituting
$z=x^2$ (a bijection of $\F_{q^2}$) and squaring, resp.\ multiplying
by $A$, we have
\[
  x\in Z_B\iff\Trq\!\big(b^2z\big)=\Trq(c_2^2\gamma^4),
  \qquad
  x\in Z_C\iff\Trq\!\big(Ad z\big)=A\big(1+\Trq(c_1\gamma)\big),
\]
with $b^2=(c_3\gamma^2+\gamma^{2q})^2\neq0$ and
$Ad=A(c_3\gamma+\gamma^q)\neq0$. Thus $Z_B=Z_C$ says that the two
surjective $\F_q$-linear functionals $z\mapsto\Trq(b^2z)$ and
$z\mapsto\Trq(Ad z)$ have a common nonempty level set; the level
sets being cosets of the kernels, the kernels agree, so the
functionals are proportional: $Ad=\mu b^2$ for some
$\mu\in\F_q^*$, and the level values satisfy
$A\big(1+\Trq(c_1\gamma)\big)=\mu\,\Trq(c_2^2\gamma^4)$. Since
$A,\mu\in\F_q$, raising $Ad=\mu b^2$ to the $q$-th power
and eliminating $\mu$ gives $b^2d^q=b^{2q}d$. Now, using
$b=\gamma d+\gamma^q(\gamma+\gamma^q)$, a direct computation shows
\[
  b^2d^q+b^{2q}d
    =\big(\gamma^2d+\gamma^{2q}d^q\big)
     \big(d^{q+1}+\Trq(\gamma^2)\big)
    =A\big(d^{q+1}+\Trq(\gamma^2)\big),
\]
so $A\neq0$ forces $d^{q+1}=\Trq(\gamma^2)$, i.e.\ $U=0$, which is
equivalent to $b^2=Ad$, whence $\mu=1$. The relation between
the level values then reads
$A\big(1+\Trq(c_1\gamma)\big)=\Trq(c_2^2\gamma^4)$, i.e.\ $V=0$.
This proves the claim.

\textbf{Case 1. $m$ even.}
If $Z_C\not\subseteq Z_B$, choose $x'$ with $C(x')=0$, $B(x')\neq0$;
then \eqref{case61} becomes $A\lambda^2+B(x')\lambda=0$ and
$\lambda=B(x')/A\in\F_q^*$ is a solution. Otherwise
$Z_C\subseteq Z_B$, and $|Z_B|=|Z_C|$ forces $Z_B=Z_C$, whence
$AC+B^2\equiv0$ by the Claim. Pick any $y\notin Z_B$; then
$AC(y)/B(y)^2=1$ and, $m$ being even,
$\Tr_2^q\!\big(AC(y)/B(y)^2\big)=\Tr_2^q(1)=0$, so \eqref{case61} has
two solutions $\lambda\in\F_q$, both nonzero since
$C(y)=B(y)^2/A\neq0$. In either case $f_1$ is not a permutation of
$\F_{q^2}$.

\textbf{Case 2. $m$ odd.}
Suppose $U\neq0$ or $V\neq0$. By the Claim $Z_B\neq Z_C$, and since
$|Z_B|=|Z_C|$, both $Z_C\setminus Z_B$ and $Z_B\setminus Z_C$ are
nonempty. For $x'\in Z_C\setminus Z_B$ the value
$\lambda=B(x')/A\in\F_q^*$ solves \eqref{case61}
(alternatively, for $y\in Z_B\setminus Z_C$, take
$\lambda=(C(y)/A)^{1/2}$), so $f_1$ is not a permutation.
\end{proof}

\begin{thm}\label{Result1}
Let $q=2^m$ with $m>1$ and let $c_1,c_2,c_3,\gamma\in\F_{q^2}$. Then
\[
  f_1(x)=x+\gamma\Trq\!\bigl(c_1x+c_2x^2+c_3x^3+x^{q+2}\bigr)
\]
is a permutation polynomial of $\F_{q^2}$ if and only if one of the following
holds:
\begin{enumerate}
  \item $\gamma=0$;
  \item $\gamma\in\F_q^{*}$, $c_3=1$, $\Trq(c_2)\neq0$ and $\Trq(c_1)=\gamma^{-1}$;
  \item $\gamma\in\F_q^{*}$, $c_3=1$, $\Trq(c_2)=0$ and $\Trq(c_1)\neq\gamma^{-1}$;
  \item $\gamma\in\F_{q^2}\setminus\F_q$, $m$ is odd, and, with
        $A:=\Trq(c_3\gamma^3+\gamma^{q+2})$,
        \[
          A\neq0,\qquad
          \Trq(\gamma^2)+(c_3\gamma+\gamma^{q})^{q+1}=0,\qquad
          A\bigl(1+\Trq(c_1\gamma)\bigr)+\Trq(c_2^2\gamma^4)=0 .
        \]
\end{enumerate}
\end{thm}
\begin{proof}
From Lemma~\ref{Re1}, we know that for $\gamma\in\F_{q^2}\setminus\F_q$ the
polynomial $f_1$ permutes $\F_{q^2}$ if and only if the conditions in item~(4)
hold. Hence we may assume $\gamma\in\F_q^{*}$ for the remainder of the proof. Similar to the Lemma~\ref{Re1}, to show $f_1$ is permutation polynomial, we need to show that the Equation~\eqref{e1} has a unique solution $u$ in $\F_q$ for every $\alpha \in \F_{q^2}$ (the derivation of Equation~\eqref{e1} did not use $\gamma \not\in \F_q$). Since $\gamma \in \F_q^*$ and $c_3=1$, Equation~\eqref{e1} reduces to the following equation
\begin{equation}\label{e2}
\begin{split}
& \Trq(c_2 ) \gamma^2u^2+\Tr_{q}^{q^2}(c_1 \alpha+ c_2 \alpha^2+ \alpha^3+\alpha^{q+2})+(1+ \gamma\Tr(c_1))u=0 \\
\end{split}
\end{equation}
Consider $c_2 \in \F_{q^2} \setminus \F_q$ and $\Trq(c_1)=\gamma^{-1}$, then Equation~\eqref{e2} has a unique solution $u \in \F_q$ satisfying $\Trq(c_2) \gamma^2u^2+\Tr_{q}^{q^2}(c_1 \alpha+ c_2 \alpha^2+c_3 \alpha^3+\alpha^{q+2})=0.$ When $c_2 \in \F_q$ and $\Trq(c_1)\neq\gamma^{-1}$, it is straightforward from Equation~\eqref{e2} to see that $f_1(x)= \alpha$ has a unique solution for every $\alpha \in \F_{q^2}$ and hence a permutation polynomial over $\F_{q^2}$.

To prove the converse, we show that there exist at least one pair $(x,u) \in \F_{q^2} \times \F_{q^2}^*$ such that
$f_1(x)+f_1(x+u)=0$, which implies that $f_1(x)$ is not a permutation over $\F_{q^2}$. When $\gamma \in \F_q^*$, $f_1(x+u)=f_1(x)$ implies that $u \in \F_q^*$, and therefore for  $(x,u) \in \F_{q^2} \times \F_{q}^*$,
\begin{equation}\label{np}
\frac{f_1(x+u) + f_1(x)}{u} \;=\; 1 + \gamma \Trq(c_1)+ \gamma \Trq(c_2)u+ \gamma \Trq(c_3)u^2+\gamma\,\Trq\!\bigl((c_3+1)(xu + x^2)\bigr).
\end{equation}

We split our analysis for $\gamma \in \F_q^*$ as follows.

\textbf{Case 1.} Let $c_3=1$. From Equation~\eqref{np}, $f_1(x+u)=f_1(x)$ for some $u \in \F_q^{*}$ if and only if $1 + \gamma \Trq(c_1)+ \gamma \Trq(c_2)u=0$.

\textbf{Subcase 1.} Suppose that $c_2 \in \F_q$ and $\Trq(c_1)=\gamma^{-1}$. Then, 
$f_1(x+u)=f_1(x)$ holds for all $x \in \F_{q^2}$ and $u \in \F_q^*$. 

\textbf{Subcase 2.} Let $c_2 \not \in \F_q$ and $\Trq(c_1)\neq\gamma^{-1}$. Then, $f_1(x+u)=f_1(x)$ for all $x \in \F_{q^2}$ and $u =\dfrac{1+\gamma \Trq(c_1)}{\gamma \Trq(c_2)}$.

\textbf{Case 2.} Let $ c_3 \neq 1$ and $\Trq(c_3)=0$. From Equation~\eqref{np}, we have
$$f_1(x+u)+f_1(x)=0 \iff 1+\gamma \Trq(c_1)+\gamma \Trq(c_2)u+ \gamma (c_3+1)\Trq(x^2)+\gamma (c_3+1)u\Trq(x)=0$$
for some $(x,u) \in \F_{q^2}\times\F_q^*$. Choose $x \in \F_{q^2}$ satisfying $\Trq(x) \neq \frac{\Trq(c_2)}{c_3+1}$ and $\Trq(x)^2 \neq \frac{1+\gamma \Trq(c_1)}{\gamma(c_3+1)}$, then we always obtain a unique $u \in \F_q^*$
$$u=\frac{1+\gamma \Trq(c_1)+ \gamma (c_3+1)\Trq(x^2)}{\gamma\Trq(c_2)+\gamma(c_3+1)\Trq(x)}.$$
We can always choose such $x$, since $\Trq$ is a surjective map on $\F_q$ and $q\geq 4$.

\textbf{Case 3.} Let $ c_3 \neq 1$ and $\Trq(c_3) \neq 0$. Substituting $u=1$ in Equation~\eqref{np}, we have
$$f_1(x+1)+f_1(x)=0 \iff 1+\gamma \Trq(c_1)+\gamma \Trq(c_2)+\gamma\Trq(c_3)+ \gamma \Trq((c_3+1)(x^2+x))=0$$
for some $x \in \F_{q^2}$.  Using Lemma~\ref{im}, the map $\phi: \F_{q^2} \mapsto \F_q$ defined as $\phi(x)=\Trq((c_3+1)(x^2+x))$ is surjective on $\F_q$ since $1+c_3 \not \in \F_q$. Thus, there will always exist an element $x \in \F_{q^2}$, which together with $u=1$ will give $f_1(x+u)=f_1(x)$.                                                                                                                                                                                                                                                                                                                                                                                                                                                                                                                                                                                                                                                                                                                                                                                                                                                                                                                                                                                                                                                                                                                                                                                                                                                                                                                                                                                                                                                                                                                                                                                                                                                                                                                                                                                                                                                                                                                                                                                                                                                                                                                                                                                                                                                                                                                                                                                                                                                                                                                                                                                                                                                                                                                                                                                                                                                                                                                                                                                                                                                                                                                                                                                                                                                                                                                                                                                                                                                                                                                           
\end{proof}

\begin{rmk}
Theorems 3.1, 3.3, 3.5, 3.6, 3.8 and 3.9 of \cite{Jiang} are special cases of the aforementioned theorem when $c_1,c_2,c_3$ are considered to be in $\F_2$. Moreover, since $(x+ x^q)x^2=x^3+x^{q+2}$,
the same holds for Theorems 5, 6, 7 and 8 of \cite{LCLL} in the case $k=1$, which
correspond to $(c_1,c_2,c_3)=(0,0,1)$, $(1,0,1)$, $(0,1,1)$ and $(1,1,1)$,
respectively. Finally, in a recent independent work \cite{HKK}, Theorem 4.3 treats
the case $n=2$; since $x^2\Tr_q^{q^2}(x)=x^3+x^{q+2}$, this result can be recovered from our Theorem~\ref{Result1} by taking $c_3=1$ and $c_1,c_2\in\F_q$.
\end{rmk}

\begin{exmp}
Let $q=2^3$ and let $a$ be a root of the irreducible polynomial
$a^6+a^4+a^3+a+1\in\F_2[x]$, so that $\F_{q^2}^*=\langle a\rangle$. Then
\[
  f_1(x)=x+(a^5 + a^4 + a^2 + a)\Trq\!\big(ax+(a^5 + a^4 + a^2 + a + 1)x^2+(a^5 + a^4 + a^3 + 1)x^3+x^{q+2}\big)
\]
is a permutation polynomial of $\F_{q^2}$ by
Theorem~\ref{Result1}, applied with $\gamma=a^5 + a^4 + a^2 + a$, $c_1=a$,
$c_2=a^5 + a^4 + a^2 + a + 1$, $c_3=a^5 + a^4 + a^3 + 1$. 
\end{exmp}

In the following two lemmas, we establish surjectivity results for maps of the form $z \mapsto \Trq\left(A\frac{z^{2q}+z}{z^2+z}\right)$, depending on $A$ is in $\F_q$ or not. 
\begin{lem}\label{L2}
Let $q=2^m$ and define $H:\F_{q^2}\setminus\F_2\to\F_q$ by
\[
  H(z)=\Trq\!\left(\frac{z^{2q}+z}{z^2+z}\right).
\]
Then $\mathrm{Im}(H)=\F_q\setminus\{1\}$.
\end{lem}

\begin{proof}
 We claim that $\mathrm{Im}(H)= \F_q \setminus \{1\}$ or equivalently, for each $\alpha \in \F_q \setminus \{1\}$ there exists $z \in \F_{q^2} \setminus \{0,1\}$ such that $H(z)=\alpha$ and $H(z) \neq 1$ for all $z$. A direct computation gives $H(z)=\dfrac{s^3(s+1)}{t(t+s+1)}$, where $s=z^q+z$ and $t=z^{q+1}$. Notice that $H(z)=0$ if and only if $s \in \{0,1\}$. 

Now, suppose that $\alpha \in \F_q^{*}$ and thus $s \in \F_q \setminus \in \{0,1\}$. Hence, $H(z)=\alpha$ has a solution if and only if 
\begin{equation}\label{leq1}
t^2+(s+1)t+\frac{s^3+s^4}{\alpha}=0.
\end{equation}
Viewing the above equation as a quadratic in $t$, there exists a solution $t \in \F_q^*$ if and only if $\Tr_2^q\left(\frac{s^3}{\alpha(s+1)}\right)=0$ for some $s \in \F_{q}\setminus \{0,1\}$.

Let there exist $(s,t) \in \F_q \setminus \{0,1\} \times \F_q^*$ satisfying Equation~\eqref{leq1}, then we need to go back and show the existence of $z \in \F_{q^2}\setminus \F_q$ satisfying $z^q+z \neq 1, z^q+z=s$ and $z^{q+1}=t$ for these pairs of $(s,t)\in \F_q \setminus \{0,1\} \times \F_q^*$. Let $x\in\F_{q^2}$ to be any root of $X^2+sX+t=0$. Then raising the quadratic equation to q power and adding it to $X^2+sX+t=0$, we get either $x \in \F_q$ or $x+x^q=s$ and $x^{q+1}=t$. 
Therefore, if we show that $x \not \in \F_q$ and $x+x^q \neq 1$, we get a $x \in \F_{q^2} \setminus \F_q$ such that $H(x)=\alpha$, for $\alpha \in \F_q^{*}$. Showing $x \not \in \F_q$, is equivalent to saying $\Tr_2^q\left(\frac{t}{s^2}\right)=1$. Moreover, if ``$x$'' in $\F_{q^2} \setminus \F_q$ satisfying $x+x^q=1$ is solution of $X^2+sX+t=0$, then
$$0=(x^2+sx+t)+(x^2+sx+t)^q=(x+x^q)^2+s(x+x^q)=1+s,$$
which is not possible since $s \not \in \{ 0,1\}$ (since $\alpha \neq 0$).

Hence, if we show there exists $s \in \F_{q} \setminus \{0,1\}$ and $t \in \F_q^*$ such that $\Tr_2^q\left(\frac{s^3}{\alpha(s+1)}\right)=0$ and $\Tr_2^q\left(\frac{t}{s^2}\right)=1$, we get the desired claim. Divide Equation~\eqref{leq1} by $s^2$, we get 
\begin{equation*}
\left(\frac{t}{s}\right)^2+\left(\frac{t}{s^2}\right)+\left(\frac{t}{s}\right)+\frac{s^2+s}{\alpha}=0.
\end{equation*}
Applying trace on the above equation, we get $\Tr_{2}^q\left(\frac{t}{s^2}\right)=\Tr_2^q\left(\frac{s^2+s}{\alpha}\right)$. Thus, if we show there exists $s \in \F_{q} \setminus \{0,1\}$ such that $\Tr_2^q\left(\frac{s^3}{\alpha(s+1)}\right)=0$ and $\Tr_2^q\left(\frac{s^2+s}{\alpha}\right)=1$, we get the desired claim. 

If $\alpha = 1$, then $\Tr_{2}^q\left(\frac{t}{s^2}\right)=\Tr_2^q\left(s^2+s\right)=0$. Thus $\alpha$ can never be equal to one. Now, put $\beta=\alpha^{-1}\in \F_q^{*}$. We want $s\in \F_q\setminus\{0,1\}$ such that
\[
\Tr_2^q\!\left(\frac{\beta s^3}{s+1}\right)=0
\quad \text{and} \quad
\Tr_2^q\!\left(\beta(s^2+s)\right)=1.
\]

Set
\[
y=\frac{1}{s+1}\in \F_q^{*},
\qquad\text{so that}\qquad
s=\frac{1}{y}+1.
\]
 Since the map $s\mapsto y=(s+1)^{-1}$ is a bijection between 
$\F_q\setminus\{0,1\}$ and $\F_q\setminus\{0,1\}$, it suffices to find $y\in \F_q^{*}\setminus \{1\}$ satisfying the corresponding conditions.

First,
\[
s^2+s
=
\left(\frac1y+1\right)^2+\left(\frac1y+1\right)
=
\frac1{y^2}+\frac1y.
\]
Hence
\[
\Tr_2^q\!\left(\beta(s^2+s)\right)
=
\Tr_2^q\!\left(\frac{\beta}{y^2}+\frac{\beta}{y}\right).
\] 
Writing $\sqrt{\beta}=\beta^{2^{m-1}}$, we obtain
\[
\Tr_2^q\!\left(\beta(s^2+s)\right)
=
\Tr_2^q\!\left(\frac{\sqrt{\beta}+\beta}{y}\right).
\]
Thus the second condition is equivalent to
\begin{equation*}
\Tr_2^q\!\left(\frac{\sqrt{\beta}+\beta}{y}\right)=1.
\end{equation*}
Next,
\[
\frac{\beta s^3}{s+1}
=
\beta\left(\frac1y+1\right)^3 y
=
\beta\left(\frac1{y^2}+\frac1y+1+y\right).
\]
Taking trace, we obtain
\[
\Tr_2^q\!\left(\frac{\beta s^3}{s+1}\right)
=
\Tr_2^q\!\left(\frac{\sqrt{\beta}+\beta}{y}\right)
+\Tr_2^q(\beta)
+\Tr_2^q(\beta y).
\]
This implies that $\Tr_2^q(\beta y)=1+\Tr_2^q(\beta)$ and now we have to solve the following system 

\[
\Tr_2^q(\beta y)=1+\Tr_2^q(\beta)
\quad \text{and} \quad
\Tr_2^q\!\left(\frac{\beta+\sqrt{\beta}}{y}\right)=1.
\]
Substituting $z=\beta y$ and $\delta = \beta(\beta+\sqrt{\beta}),$ we get
\begin{equation}\label{leq3}
 \Tr_2^q(z)=e
\quad \text{and} \quad
\Tr_2^q\!\left(\frac{\delta}{z}\right)=1
\end{equation}
where $e=1+\Tr_2^q(\beta)$. Since $y\not \in \{0,1\}$, therefore $z \not \in \{0,\beta\}$. Define
\[
N
=
\#\{z\in \F_q^{*} : \Tr_2^q(z)=e,\ \Tr_2^q(\delta/z)=1\}.
\]
Notice that $ \Tr_2^q(\beta)=e$ is not possible. Since the indicator function of the condition $\Tr_2^q(z)=\ell$ is
$\frac{1}{2}\big(1+(-1)^{\Tr_2^q(z)+\ell}\big)$, we have
\begin{align*}
N & =\sum_{z \in \F_q^*}\left(\left( \frac{1}{2}\sum_{a\in\mathbb{F}_2}(-1)^{a(\Tr_2^q(z)+e)}\right)\left( \frac{1}{2}\sum_{b\in\mathbb{F}_2}(-1)^{b(\Tr_2^q(\delta/z)+1)}\right)\right)\\
& =\frac{1}{4}\sum_{a,b\in\mathbb{F}_2}(-1)^{ae+b}
   \sum_{z\in\mathbb{F}_q^{*}}(-1)^{\Tr_2^q(az+b\delta/z)}. 
\end{align*}
Equivalently,
\[
N
=\frac{(q-1)+(-1)^{e+1}+1+(-1)^{e+1}K(\delta)}{4}=
\frac{q-(-1)^{e}+(-1)^{e+1}K(\delta)}{4},
\]
where
\[
K(\delta)
=
\sum_{z\in \F_q^{*}}
(-1)^{\Tr_2^q(z+\delta/z)}
\]
is the Kloosterman sum. For $\delta \neq 0$, by the Weil bound~\cite[Theorem 5.45]{LN} we have
\[
|K(\delta)| \le 2\sqrt{q}
\]
and thus $N \geq
\frac{q-(-1)^{e}-2\sqrt{q} }{4} \geq \frac{q-1-2\sqrt{q} }{4},$ which is at least $1$ for all $q \geq 16$. 

We need not worry about the case $\delta = 0$, because $\delta=0$ if and only if $\beta \in \{0,1\}$ but $\beta \in \F_q^*$ and $\beta =1$ if and only if $\alpha=1$, which we have already excluded.

For small $q \in \{2,4,8\}$, we verify the result experimentally using SageMath.

\end{proof}

\begin{lem}\label{L3}
Let $q=2^m$, $m>1$, and $A\in\F_{q^2}\setminus\F_q$. Then the map
$H:\F_{q^2}\setminus\F_q\to\F_q$ defined by
\[
  H(z)=\Trq\!\left(\frac{A(z^{2q}+z)}{z^2+z}\right)
\]
is surjective.
\end{lem}

\begin{proof}
On simplifying $H(z)$, we have
$$H(z):=\Trq\left(\frac{A(z^{2q}+z)}{z^2+z}\right) = \Trq(A)+(z+z^q)^2 \Trq\left(\frac{A}{z^2+z}\right).$$
Let $z^q+z=r \in \F_q$ and $B=A^q+A \in \F_{q}^*$. Write $A=B \alpha$, where $\alpha+\alpha^q=1$. Hence $\{1,\alpha\}$ is an $\F_q$-basis of $\F_{q^2}$. Put
$$d:=\alpha\,\alpha^q=\alpha+\alpha^2\in\F_q,\qquad \alpha^2=\alpha+d.$$
Also, $\Tr_{2}^{q}(d)=1$, because $X^2+X+d$ is minimal polynomial of $\alpha \in \F_{q^2} \setminus \F_q$.
 
Let $z=u+r\alpha$, where $u \in \F_q$. From $\alpha^q=1+\alpha$,
$$z^q=(u+r)+r\alpha,\qquad \Trq(z)=z+z^q=r,$$
Using $\alpha^2=\alpha+d$,
$$z^2=(u^2+r^2d)+r^2\alpha,\qquad t:=z^2+z=x+c\alpha,\quad x:=u^2+u+r^2d,\ \ c:=r^2+r.$$
A direct computation then gives
\begin{equation}\label{Hformula}
  H(z)=B+\frac{r^2Bx}{x^2+cx+c^2d}.
\end{equation}
Observe that the denominator is the norm of $t=z^2+z$, hence nonzero
precisely when $z\notin\F_2$; thus \eqref{Hformula} is valid for
every $(u,r)\in\F_q^2$ with $z=u+r\alpha\notin\{0,1\}$.

Our claim is to show that $H(z)=\beta$ has a solution $z\in \F_{q^2}\setminus\F_q$ for all $\beta \in \F_q$. First let $\beta=B$. By \eqref{Hformula} (or directly), $H(z)=B$
whenever $\Trq\!\big(A/(z^2+z)\big)=0$, i.e.\ whenever
$A/(z^2+z)\in\F_q^*$, that is,
\begin{equation}\label{ArtinSchreier}
  z^2+z=aA\qquad\text{for some } a\in\F_q^*.
\end{equation}
Equation\eqref{ArtinSchreier} has a
solution $z\in\F_{q^2}$ if and only if $\Tr_2^{q^2}(aA)=0$. By
transitivity of the trace,
\[
  \Tr_2^{q^2}(aA)
  =\Tr_2^{q}\!\big(\Trq(aA)\big)
  =\Tr_2^{q}\!\big(a\,\Trq(A)\big)
  =\Tr_2^{q}(aB),
\]
and since $B\neq0$, the map $a\mapsto\Tr_2^q(aB)$ is a surjective
$\F_2$-linear map $\F_q\to\F_2$ with kernel of size $q/2$; as
$q\ge4$, the kernel contains some $a\neq0$. For this $a$, let $z$
solve \eqref{ArtinSchreier}. Then $z\notin\F_q$:
indeed, $z\in\F_q$ would give
$z^2+z\in\F_q$, whereas $aA\notin\F_q$. Hence
$z\in\F_{q^2}\setminus\F_q$ and $H(z)=B$.

Now assume $\beta\neq B$ and set $\delta:=\beta+B\neq0$,
$\lambda:=\beta/\delta$. We seek $r\in\F_q^*$ and $u\in\F_q$ with
$H(z)=\beta$ for $z=u+r\alpha$; by \eqref{Hformula} this reads
$r^2Bx=\delta\,t^{q+1}$, i.e.\ the quadratic in $x$
\begin{equation}\label{leq4}
  x^2+Px+Q=0,
  \qquad P:=c+\frac{r^2B}{\delta}=\lambda r^2+r,
  \qquad Q:=c^2d.
\end{equation}
If $\lambda\neq0$, then $P=0$ if and only if $r=1/\lambda$; we
therefore restrict to $r\in\F_q^*\setminus\{1/\lambda\}$ (no
restriction if $\lambda=0$, where $P=r\neq0$), so that \eqref{leq4}
has two distinct roots $x$ and $x+P$.

We claim that it suffices to find such an $r$ with
\begin{equation}\label{condIII}
  \text{(i)}\ \ \Tr_2^q(P)=1
  \qquad\text{and}\qquad
  \text{(ii)}\ \ \Tr_2^q\!\Big(\frac{Q}{P^2}\Big)=0.
\end{equation}
Indeed, suppose $r$ satisfies (i) and (ii). By (ii), both roots of
\eqref{leq4} lie in $\F_q$. Since the roots are $x$ and $x+P$, their
absolute traces differ by $\Tr_2^q(P)=1$; hence exactly one root, say
$x^\ast$, satisfies $\Tr_2^q(x^\ast+r^2d)=0$. By
\cite[Theorem~2.25]{LN}, there then exists $u\in\F_q$ with
$u^2+u=x^\ast+r^2d$, i.e.\ $x^\ast=u^2+u+r^2d$ is of the required
form. For this pair $(u,r)$, the element $z=u+r\alpha$ lies in
$\F_{q^2}\setminus\F_q$ (as $r\neq0$) and satisfies $H(z)=\beta$.

It remains to exhibit such an $r$. For (i): since
$\Tr_2^q(\lambda r^2)=\Tr_2^q(\sqrt{\lambda}\,r)$, we have
\begin{equation}\label{Cond1}
  \Tr_2^q(P)=\Tr_2^q\big((\sqrt{\lambda}+1)\,r\big),
\end{equation}
and $\sqrt{\lambda}+1\neq0$, because $\lambda=1$ would force $B=0$.
For (ii): as $Q=c^2d$ and, for $r\neq0$,
$\dfrac{c}{P}=\dfrac{r+1}{\lambda r+1}$, we get
\begin{equation}\label{Cond2}
  \Tr_2^q\!\Big(\frac{Q}{P^2}\Big)
  =\Tr_2^q\!\left(\Big(\sqrt{d}\,\frac{r+1}{\lambda r+1}\Big)^{2}\right)
  =\Tr_2^q\!\left(\sqrt{d}\,\frac{r+1}{\lambda r+1}\right).
\end{equation}

Let $\beta=0$ (i.e.\ $\lambda=0$).Here \eqref{Cond1} and \eqref{Cond2} become $\Tr_2^q(r)=1$ and
$\Tr_2^q\!\big(\sqrt{d}\,(r+1)\big)=0$, the latter being equivalent to
$\Tr_2^q(\sqrt{d}\,r)=\Tr_2^q(\sqrt{d})=\Tr_2^q(d)=1$. Consider the
$\F_2$-linear map
\[
  L:\F_q\to\F_2^2,
  \qquad
  L(r)=\big(\Tr_2^q(r),\,\Tr_2^q(\sqrt{d}\,r)\big).
\]
Since every $\F_2$-linear functional of $\F_q$ is of the form
$r\mapsto\Tr_2^q(ar)$ for a unique $a\in\F_q$, the two coordinate
functionals of $L$ (given by $a=1$ and $a=\sqrt{d}$) are linearly
independent over $\F_2$ if and only if $d\neq1$; in that case $L$ is
surjective and $(1,1)$ is attained. If $d=1$, the two conditions
coincide with $\Tr_2^q(r)=1$, which has $q/2$ solutions; and $d=0$ is
impossible, as $\Tr_2^q(d)=1$. In every case there exists $r$ with
$\Tr_2^q(r)=\Tr_2^q(\sqrt{d}\,r)=1$; note that $r\neq0$
automatically, since $L(0)=(0,0)$.

Suppose $\beta\neq0$ (i.e.\ $\lambda\neq0$). Recall also that $\lambda\neq1$, since $B\neq0$. Let $N_r$ denote the
number of $r\in\F_q^*\setminus\{1/\lambda\}$ satisfying \eqref{Cond1}
and \eqref{Cond2}. Therefore,
\begin{align*}
 N_r &=\sum_{r\in\F_q^*\setminus\{1/\lambda\}}
   \left(\frac{1}{2}\sum_{a\in\F_2}
     (-1)^{a\big(\Tr_2^q((\sqrt{\lambda}+1)r)+1\big)}\right)
   \left(\frac{1}{2}\sum_{b\in\F_2}
     (-1)^{b\,\Tr_2^q\big(\sqrt{d}\,\frac{r+1}{\lambda r+1}\big)}\right)\\
 &=\frac{1}{4}\sum_{a,b\in\F_2}(-1)^{a}
   \sum_{r\in\F_q^*\setminus\{1/\lambda\}}
   (-1)^{\Tr_2^q\big(a(\sqrt{\lambda}+1)r
     +b\,\sqrt{d}\,\frac{r+1}{\lambda r+1}\big)}
  \;=\;\frac{(q-2)-S_a+S_b-S_{ab}}{4},
\end{align*}
where $S_a$, $S_b$, $S_{ab}$ denote the inner sums for
$(a,b)=(1,0)$, $(0,1)$, $(1,1)$, respectively.

Since $\sqrt{\lambda}+1\neq0$, the sum $S_a$ is a nontrivial additive
character sum over $\F_q$ with the two points $r=0$ and $r=1/\lambda$
removed; hence $|S_a|\le2$. Similarly, as
$r\mapsto\frac{r+1}{\lambda r+1}$ is injective on
$\F_q\setminus\{1/\lambda\}$ and $\sqrt{d}\neq0$, we get $|S_b|\le2$.

For $S_{ab}$, substitute $T=\lambda r+1$, which maps
$\F_q^*\setminus\{1/\lambda\}$ bijectively onto
$\F_q\setminus\{0,1\}$. A direct computation gives
\begin{align*}
 S_{ab}
 &=(-1)^{\Tr_2^q(\varepsilon)}
   \sum_{T\in\F_q\setminus\{0,1\}}
   (-1)^{\Tr_2^q\big(\frac{\sqrt{\lambda}+1}{\lambda}\,T
     +\frac{\sqrt{d}(1+\lambda)}{\lambda}\,\frac{1}{T}\big)},
 \qquad
 \varepsilon:=\frac{\sqrt{\lambda}+\sqrt{d}+1}{\lambda},
\end{align*}
and, completing the sum over $\F_q^*$ by adding and subtracting the
term $T=1$ (whose exponent is
$\Tr_2^q(\varepsilon+\sqrt{d})=\Tr_2^q(\varepsilon)+1$),
\[
 S_{ab}
 =(-1)^{\Tr_2^q(\varepsilon)}
   \,K\!\left(\frac{\sqrt{\lambda}+1}{\lambda},\,
     \frac{\sqrt{d}(1+\lambda)}{\lambda}\right)
   +1.
\]
Both parameters are nonzero: $\sqrt{\lambda}+1\neq0$ as noted, and
$\sqrt{d}(1+\lambda)\neq0$ since $d\neq0$ and $\lambda\neq1$. Hence
the Weil bound~\cite[Theorem~5.45]{LN} gives
$|K(\cdot,\cdot)|\le2\sqrt{q}$, and therefore
$|S_{ab}|\le2\sqrt{q}+1$. Altogether,
\[
 N_r\;\ge\;\frac{(q-2)-2-2-(2\sqrt{q}+1)}{4}
 \;=\;\frac{q-2\sqrt{q}-7}{4},
\]
which is positive for $q\ge16$; since $N_r$ is an integer, $N_r\ge1$.
For $q\in\{4,8\}$, the statement of the lemma has been verified
directly by computer using SageMath.

Therefore an admissible $r$ exists, and by the claim above,
$H(z)=\beta$ for some $z\in\F_{q^2}\setminus\F_q$. Together with the
cases $\beta=B$ and $\beta=0$, this shows that $H$ is surjective.
\end{proof}

In the following result we give a characterization of second family of permutation polynomials, over $\F_{q^2}$. 

\begin{thm}\label{Result2}
Let $q=2^m$, $m>1$ $c_1, c_2, c_3, \gamma \in \F_{q^2}$. Then
$$f_2(x)=x+\gamma \Tr_{q}^{q^2} (c_1 x+ c_2 x^2+ c_3 x^3+ c_3x^{q+2}+x^{2q-1}),$$
is a permutation polynomial of $\F_{q^2}$ if and only if one of the following holds,
\begin{enumerate}
\item $\gamma=0$,
\item $\gamma \in \F_{q}^*$, $c_2, c_3 \in \F_{q}$ and $1+\gamma\Trq(c_1)=\gamma$. 
\item $\gamma \in \F_{q^2} \setminus \F_{q}$, $m$ is odd, $\gamma^{2q-1} \in \F_q$, $\Trq(c_2\gamma^q)=0$, $c_3=0$ and $1+\Trq(c_1\gamma)=\gamma^{2q-1}$,
\end{enumerate}
\end{thm}
\begin{proof}
We first show that $f_2$ is a permutation polynomial over $\F_{q^2}$ by splitting our analysis in two cases depending on $\gamma \in \F_q^*$ or $\gamma \not \in \F_q$.

Let $\gamma \in \F_q^*$. To show $f_2(x)$ is a permutation, it is enough to show $f_2(x)=\alpha$ has a unique solution for all $\alpha \in \F_{q^2}$. If $\alpha=0$, then $x= \gamma \Tr_{q}^{q^2} (c_1 x+ c_2 x^2+ c_3x^3+c_3x^{q+2}+x^{2q-1}) \in \F_q$. For $c_2,c_3 \in \F_q$ and $1+\gamma\Trq(c_1)=\gamma$, this further implies that $$f_2(x)=x(1+\gamma \Tr_{q}^{q^2} (c_1))=x(1+\gamma(1+\gamma^{-1}))=x \gamma=0,$$ or equivalently $x=0$. Therefore, for $\alpha=0$, the equation $f_2(x)=0$ has a unique solution $x=0$ and conversely, if $x=0$ is a solution of $f(x)=\alpha$ then $\alpha$ has to be zero. 
Suppose that $\alpha \neq 0$ and write $\frac{x+\alpha}{\gamma} = u \in \F_q$. Then, substituting $x=\alpha+\gamma u$ in $f_2(x)=\alpha$,  multiplying by $\gamma x^{q+1}$ and assuming $c_2, c_3 \in \F_{q}$ with $\Tr_{q}^{q^2}(c_1)=\gamma^{-1}+1$, we obtain
\begin{align}\label{E2}
 \begin{cases}
  & \gamma^3 u^3+(\gamma\Trq(\alpha)+\gamma^2\Trq(c_1\alpha^q+c_2 \alpha^2+ \alpha+ c_3\alpha^3+ c_3\alpha^{2q+1}))u^2+\\
  & (\alpha^{q+1}+\gamma\Trq(c_1\alpha^2+c_2 \alpha^3+c_2 \alpha^{q+2}+ \alpha^2+ c_3\alpha^4))u\\
  & 
  + \Trq(c_1 \alpha^{q+2}+c_2 \alpha^{q+3}+c_3\alpha^{q+4}+c_3\alpha^{2q+3}+\alpha^{3})=0.
 \end{cases}
\end{align}
For $\alpha \in \F_q$, the above equation further simplifies to
$$\gamma^3 u^3+ \gamma (\gamma+1) \alpha u^2+ \gamma \alpha^2 u+ (\gamma^{-1}+1)\alpha^3=(\alpha+\gamma u)^2(\gamma u+(\gamma^{-1}+1)\alpha)=0,$$
and it has a unique solution $u=\dfrac{(\gamma+1)\alpha}{\gamma^2}$, since the other root $u=\frac{\alpha}{\gamma}$ give $x=0$. Next, we consider the case when 
$\alpha \in \F_{q^2} \setminus \F_q$. Denote the coefficients of Equation~\eqref{E2} as follows
\begin{equation*}
 \begin{cases}
  A &=\gamma^3 \\
  B &= (\gamma+\gamma^2)(\alpha+\alpha^q)+c_2\gamma^2(\alpha+\alpha^q)^2+c_3\gamma^2(\alpha+\alpha^q)^3+ c_1\gamma^2(\alpha+\alpha^q)+(\gamma+\gamma^2)\alpha \\
  C &= \gamma c_3(\alpha+\alpha^q)^4+c_2\gamma (\alpha+\alpha^q)^3+ (\gamma c_1+\gamma)(\alpha+\alpha^q)^2+(\gamma+1)\alpha^{2q}+\alpha^{q+1} \\
  D&= c_2 \alpha^{q+1}(\alpha+\alpha^q)^2+ (c_3\alpha^{q+1}+1)(\alpha+\alpha^q)^3+c_1\alpha^{q+1}(\alpha+\alpha^q)+\gamma^{-1} \alpha^{2q+1}+\alpha^{q+2},
 \end{cases}
\end{equation*}
and observe that all the coefficients are in $\F_q$ and we have the following relations between the coefficients
\begin{equation*}
 \begin{cases}
  C&=\gamma (\alpha^{q+1}+\alpha^2+\alpha^{2q})+\gamma^{-1} (\alpha+\alpha^{q})B\\
  D&=\alpha^{q+1} \gamma^{-2} B+\alpha^3+\alpha^{3q}
 \end{cases}
\end{equation*}
On substituting $u=y+B/A$ in Equation~\eqref{E2}, we get
\begin{equation}\label{nq}
 y^3+ay+b=0,
\end{equation}
where, 
\begin{equation*}
 \begin{cases}
  a & =\dfrac{B^2+\gamma^2(\alpha+\alpha^q)B+\gamma^4(\alpha^{q+1}+\alpha^2+\alpha^{2q})}{\gamma^6}\\
  b & =\gamma^{-1}(\alpha+\alpha^q)a.
 \end{cases}
\end{equation*}
For $a=0$, Equation~\eqref{nq} has a unique solution $y=0$ in $\F_q$, that is $u=B/A$. Let $a \neq 0$ then using Lemma~\ref{L1}, Equation~\eqref{nq} has a unique solution if $\Tr_{2}^{q}\left(\frac{a \gamma^2}{(\alpha+\alpha^q)^2}\right) \neq \Tr_{2}^{q}(1)$. Consider
\begin{align*}
 \Tr_{2}^{q}\left(\frac{a \gamma^2}{(\alpha+\alpha^q)^2}\right) & =\Tr_{2}^{q}\left(\frac{B^2+\gamma^2(\alpha+\alpha^q)B+\gamma^4(\alpha^{q+1}+\alpha^2+\alpha^{2q})}{\gamma^4(\alpha+\alpha^q)^2}\right) \\
 & = \Tr_{2}^{q}\left(\frac{B^2+\gamma^2(\alpha+\alpha^q)B+\gamma^4(\alpha^{q+1})}{\gamma^4(\alpha+\alpha^q)^2}\right)+\Tr_{2}^{q}(1)\\
 &= \Tr_{2}^{q}\left(\frac{\alpha^{q+1}}{\alpha^2+\alpha^{2q}}\right)+\Tr_{2}^{q}(1)\\
 & = 1+\Tr_{2}^{q}(1).
\end{align*}
The last equality holds because the equation $x^2 +(\alpha+\alpha^q)x+\alpha^{q+1} = 0$ has two solutions $\alpha, \alpha^q \in \F_{q^2} \setminus \F_q$. Therefore, $f_2(x)=\alpha$ has a unique solution for all $\alpha \in \F_{q^2}$ and hence, $f_2(x)$ is a permutation polynomial when $\gamma, c_2, c_3 \in \F_q$ and $\Trq(c_1)=\gamma^{-1}+1$. 

Let $\gamma \not \in \F_q$, $m$ is odd, $\gamma^{2q-1}=\gamma^{2-q}$, $c_3=0$, $\Trq(c_2\gamma^q)=0$ and $1+\Trq(c_1\gamma)=\gamma^{2q-1}$. Here, to show $f_2(x)$ is a permutation polynomial, we show that $f_2(x+v)+f_2(x) \neq 0$ for all $x \in \F_{q^2}$ and $v \in \F_{q^2}^*$. On the contrary, assume that $f_2(x+v)+f_2(x)=0$ for some $x \in \F_{q^2}$ and $v \in \F_{q^2}^*$. Then,
\begin{equation}\label{pp2}
 f_2(x+v)+f_2(x) =v+\gamma \Trq(c_1v+c_2v^2+(x+v)^{2q-1}+x^{2q-1})=0
\end{equation}
In particular $v=\gamma\,\Trq(c_1v+c_2v^2+(x+v)^{2q-1}+x^{2q-1})\in\gamma\F_q$, so we may write
$v=b\gamma$ with $b\in\F_q^*$.

First suppose $x\in\{0,v\}$. Then \eqref{pp2} reads $f_2(v)=0$, i.e.
\[
  v+\gamma\,\Trq\!\big(c_1v+c_2v^2+v^{2q-1}\big)=0.
\]
Since $v^{2q-1}=b^{2q-1}\gamma^{2q-1}\in\F_q$, its trace vanishes, and
substituting $v=b\gamma$,
\begin{align*}
  f_2(v)=0
  &\iff b\gamma\big(1+\Trq(c_1\gamma)\big)+b^2\gamma\,\Trq(c_2\gamma^2)=0\\
  &\iff b\gamma^{2q}+b^2\gamma^{2q}\big(c_2\gamma^q+c_2^q\gamma\big)=0
  \iff b\gamma^{2q}+b^2\gamma^{2q}\,\Trq(c_2\gamma^q)\,\gamma
  =b\gamma^{2q}=0.
\end{align*}
As $b\gamma^{2q}\neq0$, this is a contradiction. Now let $x\notin\{0,v\}$ and set $z:=x/v\in\F_{q^2}\setminus\{0,1\}$.
Since
\[
  (x+v)^{2q-1}+x^{2q-1}
  =v^{2q-1}\big((z+1)^{2q-1}+z^{2q-1}\big)
  =v^{2q-1}\,\frac{z^{2q}+z}{z^2+z},
\]
and $v^{2q-1}=b^{2q-1}\gamma^{2q-1}\in\F_q$, equation \eqref{pp2}
becomes, after substituting $v=b\gamma$ and dividing by $b\gamma$,
\begin{equation}\label{pp3}
  \gamma^{2q-1}\,b^{2q-2}\,
  \Trq\!\Big(\frac{z^{2q}+z}{z^2+z}\Big)
  +1+\Trq(c_1\gamma)+b\,\Trq(c_2\gamma^2)=0.
\end{equation}
Using the given constraints, Equation \eqref{pp3} reduces to
\[
  \Trq\!\Big(\frac{z^{2q}+z}{z^2+z}\Big)=b^{2-2q}=b^0=1
  \qquad(b\in\F_q^*),
\]
contradicting Lemma~\ref{L2}, by which $1$ is not in the image of this
map on $\F_{q^2}\setminus\{0,1\}$. Hence $f_2(x+v)+f_2(x)=0$ has no
solution, and $f_2$ is a permutation polynomial.

We show that if none of the conditions (2)--(3)
holds (and $\gamma\neq0$), then $f_2$ is not a permutation
polynomial. We treat $\gamma\in\F_q^*$ (Cases 1(a)--1(e): some clause of
condition (2) fails) and $\gamma\in\F_{q^2}\setminus\F_q$ (Case 2:
some clause of condition (3) fails) separately; in each case we
exhibit $x\in\F_{q^2}$, $w\in\F_{q^2}^*$ with $f_2(x+w)=f_2(x)$. In the following table we list all the necessity cases when $\gamma \in \F_q^*$. Observe that $1+\gamma\Trq(c_1)=0$ implies $1+\gamma\Trq(c_1)\neq \gamma$ or in other words if $1+\gamma\Trq(c_1)=\gamma$ then $1+\gamma\Trq(c_1)\neq 0$.

\begin{table}[h!]
\centering
\renewcommand{\arraystretch}{1.4}
\begin{tabular}{c c c c}
  \toprule
  $c_2$ & $c_3$ & $1+\gamma\Trq(c_1)$ & Case\\
  \midrule
  $\in\F_q$    & $\in\F_q$    & $=\gamma$              & --- (condition (2)) \\
  $\in\F_q$    & $\in\F_q$    & $=0$                   & 1(b) \\
  $\in\F_q$    & $\in\F_q$    & $\neq0,\ \neq\gamma$   & 1(d) \\
  $\in\F_q$    & $\notin\F_q$ & $=0$                   & 1(b) \\
  $\in\F_q$    & $\notin\F_q$ & $\neq0$                & 1(e) \\
  $\notin\F_q$ & any          & $=0$                   & 1(c) \\
  $\notin\F_q$ & any          & $\neq0$                & 1(a) \\
  \bottomrule
\end{tabular}
\caption{Coverage of the necessity cases for $\gamma\in\F_q^{*}$.
The first row is condition~(2) of Theorem~\ref{Result2}, where $f_2$
is a permutation polynomial.}
\label{tab:coverageI}
\end{table}

Note that for $x\in\F_q$,
\begin{equation}\label{Fqeval}
  f_2(x)=(1+\gamma\Trq(c_1))x+\gamma\Trq(c_2)\,x^2.
\end{equation}
Moreover, if $w\in\F_q^*$,
$x\notin\{0,w\}$, and $z:=x/w$, then $z\in\F_{q^2}\setminus\{0,1\}$
and
\begin{equation}\label{ratid}
  (x+w)^{2q-1}+x^{2q-1}=w\,\frac{z^{2q}+z}{z^2+z}=:w\,R(z),
\end{equation}
using $w^{2q-1}=w$. Finally, by the proof of Lemma~\ref{L2}
(the case $s\in\{0,1\}$), 
\begin{equation}\label{L2cor}
  \Trq(z)\in\{0,1\}\ \Longrightarrow\ \Trq\!\big(R(z)\big)=0.
\end{equation}

\textbf{Case 1.}  First suppose that $\gamma\in\F_q^*$.

\textbf{Case 1(a).} Let $c_2\notin\F_q$ and $1+\gamma\Trq(c_1)\neq 0$. By \eqref{Fqeval}, $f_2$ vanishes at the two distinct points $x=0$ and $x=(1+\gamma\Trq(c_1))/\big(\gamma\Trq(c_2)\big)$ of $\F_q$; not a permutation.

\textbf{Case 1(b).} Suppose that $c_2\in\F_q$ and $1+\gamma\Trq(c_1)=0$. By \eqref{Fqeval}, $f_2(x)=0$ for every $x\in\F_q$; not a permutation.

\textbf{Case 1(c).} Let $c_2\notin\F_q$ and $1+\gamma\Trq(c_1)=0$.
Take any $w\in\F_q^*$ and $x\notin\{0,w\}$; using \eqref{ratid}, equation $f_2(x+w)+f_2(x)=0$ becomes, after
dividing by $\gamma w$,
\[
  \Trq(c_2)\,w+\Trq(c_3)\Trq(x)\,w+\Trq\!\big(R(z)\big)=0,
  \qquad z=x/w.
\]
If $\Trq(c_3)=0$: choose $w\in\F_q^*$ with $\Trq(c_2)w\neq1$
(possible as $q\ge4$); by Lemma~\ref{L2} there is
$z\in\F_{q^2}\setminus\{0,1\}$ with $\Trq(R(z))=\Trq(c_2)w$, and
$x=wz$ gives a collision. If $\Trq(c_3)\neq0$: put
$w:=\Trq(c_2)/\Trq(c_3)\in\F_q^*$ (nonzero as $c_2\notin\F_q$) and
choose $x$ with $\Trq(x)=w$; then $\Trq(z)=1$, so
$\Trq(R(z))=0$ by \eqref{L2cor}, and the left-hand side equals
$\Trq(c_2)w+\Trq(c_2)w=0$; again a collision.

\textbf{Case 1(d).} Next, let $c_2\in\F_q$, $c_3\in\F_q$, $1+\gamma \Trq(c_1)\neq0$, and $1+\gamma \Trq(c_1) \neq\gamma$.
For $w\in\F_q^*$, $x\notin\{0,w\}$, equation $f_2(x+w)+f_2(x)=0$ reduces to $\Trq(R(z))=(1+\gamma\Trq(c_1))/\gamma$. Since $(1+\gamma\Trq(c_1))/\gamma\notin\{0,1\}$,
Lemma~\ref{L2} provides such a $z$, and $x=wz$ gives a collision.

\textbf{Case 1(e).} Suppose that $c_2\in\F_q$, $c_3\notin\F_q$, $1+\gamma\Trq(c_1)\neq0$. Equation $f_2(x+w)+f_2(x)=0$ is now equivalent to $1+\gamma\Trq(c_1)+\gamma\Trq(c_3)\Trq(x)\,w+\gamma\Trq(R(z))=0$. Let
$w\in\F_q^*$ with $w^2=(1+\gamma\Trq(c_1))/\big(\gamma\Trq(c_3)\big)$ and choose $x$ with $\Trq(x)=w$; then
$\Trq(z)=1$, so $\Trq(R(z))=0$ by \eqref{L2cor}, and
$1+\gamma\Trq(c_1)+\gamma\Trq(c_3)w^2=0$; a collision.

\smallskip
Cases  1(a)-- 1(e) cover every violation of condition (2) in the statement; hence for
$\gamma\in\F_q^*$ we must have $c_2,c_3\in\F_q$ and $1+\gamma\Trq(c_1)=\gamma$.

\textbf{Case 2.}
Let \(\gamma\notin\F_q\) and suppose \(f_2(x+w)=f_2(x)\), so that \(w=s\gamma\) with \(s\in\F_q^*\). Writing \(x=wt\), \(t\in\F_{q^2}\setminus\{0,1\}\), we obtain

\begin{equation}\label{eqgen}
1+\Tr_q(c_1\gamma)+s\Tr_q(c_2\gamma^2)
+\Tr_q\!\big(\gamma^{2q-1}R(t)\big)+s^2G(t)=0,
\end{equation}
where
$$
G(t)=\Tr_q\!\left(c_3\gamma^3(t^2+t+1)
+c_3\gamma^{q+2}(t^2+t^q+1)\right).
$$

\textbf{Case 2(a).} Let \(c_3=0\).
Then

$$
\Tr_q\!\big(\gamma^{2q-1}R(t)\big)
=1+\Tr_q(c_1\gamma)+s\Tr_q(c_2\gamma^2).
$$

If \(\gamma^{2q-1}\notin\F_q\), Lemma~\ref{L3} gives an admissible \(t\) for every fixed \(s\), hence a collision. Therefore $\gamma^{2q-1}\in\F_q$ is necessary. In this case
$$
\Tr_q(R(t))
=\rho(s):=
\frac{1+\Tr_q(c_1\gamma)
+s\gamma^{2-q}\Tr_q(c_2\gamma^q)}
{\gamma^{2q-1}}.
$$

By Lemma~\ref{L2}, every value of \(\F_q\setminus\{1\}\) occurs as
\(\Tr_q(R(t))\). Hence, if \(\Tr_q(c_2\gamma^q)\neq0\), one may choose
\(s\in\F_q^*\) with \(\rho(s)\neq1\), producing a collision. Thus
$\Tr_q(c_2\gamma^q)=0.$

Then \(\rho(s)\) is constant, and another application of Lemma~\ref{L2}
shows that a collision occurs unless

$$
1+\Tr_q(c_1\gamma)=\gamma^{2q-1}.
$$

Therefore these three conditions are necessary.

\medskip
\textbf{Case 2(b).} Let \(c_3\neq0\).
Taking \(t\in\F_q\setminus\{0,1\}\), so that \(R(t)=1\), gives

$$
s^2\Tr_q(\gamma)\Tr_q(c_3\gamma^2)(t^2+t+1)
+s\Tr_q(c_2\gamma^2)
+1+\Tr_q(c_1\gamma)+\Tr_q(\gamma^{2q-1})=0.
$$

Assume first that \(\Tr_q(c_3\gamma^2)\neq0\). Put

$$
E=\Tr_q(c_2\gamma^2),\qquad
\mu=1+\Tr_q(c_1\gamma)+\Tr_q(\gamma^{2q-1}),
$$

and

$$
\delta=
\frac{E}{\Tr_q(\gamma)\Tr_q(c_3\gamma^2)}
+\left(
\frac{\mu}{\Tr_q(\gamma)\Tr_q(c_3\gamma^2)}
\right)^{1/2}.
$$
The quadratic in \(t\) has an admissible root whenever
$\Tr_2^q(\delta/s)=\Tr_2^q(1),$ apart from at most two exceptional values of \(s\). Hence, if \(\delta\neq0\), there is a collision for \(q\ge8\). If \(\delta=0\) and \(m\) is even, the same conclusion follows for any \(s \in \F_q^*\).

Thus the only remaining case is \(m\) odd, \(\delta=0\), i.e.
$$
E^2=
\mu\,\Tr_q(\gamma)\Tr_q(c_3\gamma^2).
$$

We now take \(t\in\F_{q^2}\setminus\F_q\) and using Equation~\eqref{eqgen}, write

$$
s^2G(t)+sE+C(t)=0,
\qquad
C(t)=1+\Tr_q(c_1\gamma)
+\Tr_q\!\big(\gamma^{2q-1}R(t)\big).
$$

Let

$$
Z_G=\{t:G(t)=0\},
\qquad
Z_C=\{t:C(t)=0\},
$$

where \(t\in\F_{q^2}\setminus\F_q\).

If \(E=0\), then a collision exists unless \(Z_G\) and \(Z_C\) form a
disjoint partition of \(\F_{q^2}\setminus\F_q\). Since \(G\) is an
affine linearised polynomial of degree \(2q\), $|Z_G|\le2q,$ so \(|Z_C|\ge q^2-3q\). On the other hand, Lemmas~\ref{L2}--\ref{L3}
give $|Z_C|\le4q.$ Hence \(q^2-3q\le4q\), which is impossible for \(q\ge8\). Therefore a collision always exists when \(E=0\).

Now assume \(E\neq0\). If no collision exists, then necessarily

$$
Z_G=Z_C.
$$

Suppose first \(Z_G=\emptyset\). Lemmas~\ref{L2}--\ref{L3} then force
$$
\gamma^{2q-1}\in\F_q,
\qquad
1+\Tr_q(c_1\gamma)=\gamma^{2q-1}.
$$

Moreover \(c_3\in\F_q\); otherwise the two independent \(\F_q\)-linear
forms

$$
P(t)=\Tr_q(\gamma t),\qquad
Q(t)=\Tr_q(c_3\gamma^2t^2)
$$

make \(G\) attain the value \(0\) outside \(\F_q\).
Consequently

$$
G(t)=c_3\Tr_q(\gamma)^3(\pi^2+\pi+1),
\qquad
\pi=\frac{\Tr_q(\gamma t)}{\Tr_q(\gamma)},
$$

which never vanishes because \(m\) is odd. Therefore the condition,

$$
G(t)s^2+Es+C(t)=0
$$

have a root \(s\in\F_q^*\) reduces to

$$
\Tr_2^q\!\big((\pi^2+\pi+1)(1+\rho)\big)=0,
\qquad
\rho=\Tr_q(R(t)).
$$

Comparing \(t\) and \(t^q\), and writing
\(\tau=\Tr_q(t)\), this is satisfied once $
\Tr_2^q\!\big(\rho(\tau^2+\tau)\big)=1.
$

By the proof of Lemma~\ref{L2}, $\rho=\tau^3(\tau+1)/\big(\nu(\nu+\tau+1)\big)$
with $\nu=t^{q+1}$. Take $\tau\neq1$ and write $\nu=(\tau+1)\psi$, so that
$\nu(\nu+\tau+1)=(\tau+1)^2(\psi^2+\psi)$; the condition
$\Tr_2^q\big(\rho(\tau^2+\tau)\big)=1$ then reads
\[
  \Tr_2^q\Big(\frac{\tau^4}{\psi^2+\psi}\Big)=1 .
\]
Moreover $t\in\F_{q^2}\setminus\F_q$ with $\Tr_q(t)=\tau$ and $t^{q+1}=\nu$
exists exactly when $X^2+\tau X+\nu$ is irreducible over $\F_q$, that is when
\[
  \Tr_2^q\Big(\frac{(\tau+1)\psi}{\tau^2}\Big)=1 ,
\]
and then $\psi\neq0,1$ automatically. Such a pair $(\tau,\psi)$ exists. Fix $\psi\in\F_q\setminus\{0,1\}$ and put
$a=(\psi^2+\psi)^{-1/4}$ and $b=\psi+\psi^{1/2}$, both nonzero. Since
$\Tr_2^q(x^2)=\Tr_2^q(x)$, the two conditions become $\Tr_2^q(a\tau)=1$ and
$\Tr_2^q(b/\tau)=1$. Expanding the indicators, the number of
$\tau\in\F_q^*$ satisfying both is $\big(q+1+K(a,b)\big)/4$, which by the Weil
bound is at least $(\sqrt q-1)^2/4\ge2$ for $q\ge16$; hence some admissible
$\tau$ differs from $1$. For $q=8$ this is checked directly. This contradicts
$Z_G=\emptyset$.

Finally suppose \(Z_G\neq\emptyset\). Then \(c_3\notin\F_q\), so \(G\)
takes every value of \(\F_q\) exactly \(q\) times and $
|Z_G|=|Z_C|=q$. Decompose

$$
\F_{q^2}\setminus\F_q
=\bigcup_{\tau\in\F_q^*}S_\tau,
\qquad
S_\tau=\{t:\Tr_q(t)=\tau\}.
$$

Fix $\alpha\in\F_{q^2}$ with $\alpha+\alpha^q=1$, so that every $t\in S_\tau$ is
uniquely $t=\tau\alpha+a$ with $a\in\F_q$. Expanding $G$ as
\[
 G(t)=\Trq(\gamma)\Trq(c_3\gamma^2)+\Trq(\gamma)\Trq(c_3\gamma^2t^2)
       +\Trq(c_3\gamma^2)\Trq(\gamma t),
\]
and using $\Trq(c_3\gamma^2a^2)=a^2\Trq(c_3\gamma^2)$ and
$\Trq(\gamma a)=a\Trq(\gamma)$ for $a\in\F_q$, we obtain
\[
 G(t)=\Trq(\gamma)\Trq(c_3\gamma^2)\big(1+\phi(\tau)+a^2+a\big),
 \qquad
 \phi(\tau)=\tau^2\,\frac{\Trq(c_3\gamma^2\alpha^2)}{\Trq(c_3\gamma^2)}
            +\tau\,\frac{\Trq(\gamma\alpha)}{\Trq(\gamma)} .
\]
Hence $G(t)=0$ reads $a^2+a=1+\phi(\tau)$, which has either no root or exactly
two, and is solvable precisely when $\Tr_2^q\big(1+\phi(\tau)\big)=0$, that is
when $\Tr_2^q(\phi(\tau))=\Tr_2^q(1)=1$ as $m$ is odd. Since
$\Tr_2^q(\tau^2c)=\Tr_2^q(\tau\sqrt c)$ for $c\in\F_q$, this last condition is
$\Tr_2^q(\omega\tau)=1$, where
\[
 \omega=\left(\frac{\Trq(c_3\gamma^2\alpha^2)}{\Trq(c_3\gamma^2)}\right)^{1/2}
        +\frac{\Trq(\gamma\alpha)}{\Trq(\gamma)} .
\]
Therefore $|Z_G\cap S_\tau|\in\{0,2\}$, and as $|Z_G|=q$ the set $Z_G$ meets
exactly $q/2$ of the sets $S_\tau$; in particular $\omega\neq0$. Since
$Z_C=Z_G$, the same holds for $Z_C$.

If \(\gamma^{2q-1}\in\F_q\), then \(Z_C\) is invariant under
\(t\mapsto t+1\) and \(t\mapsto t^q\). For \(\tau\neq1\), its orbits in
\(S_\tau\) have size \(4\), contradicting
\(|Z_C\cap S_\tau|\le2\); the case \(\tau=1\) is also impossible.
Hence \(\gamma^{2q-1}\notin\F_q\).

Put $B=\Trq(\gamma^{2q-1})\neq0$ and take $\alpha=\gamma^{2q-1}/B$ above, so
that $\alpha+\alpha^q=1$ and $d:=\alpha^{q+1}$ satisfies $\Tr_2^q(d)=1$, in the
notation of the proof of Lemma~\ref{L3}. If \(1+\Tr_q(c_1\gamma)=B\), Lemma~\ref{L3} gives $|Z_C|=q-2,$ contrary to \(|Z_C|=q\). Otherwise, set $\lambda=\mu_0/(\mu_0+B)$, where $\mu_0=1+\Trq(c_1\gamma)$; this
is the parameter $\lambda$ of the proof of Lemma~\ref{L3} for the value
$\beta=\mu_0$, and $\lambda\neq1$ since $B\neq0$. By that proof,
$|Z_C\cap S_\tau|=2$ requires both
\[
  \Tr_2^q\big((\sqrt\lambda+1)\tau\big)=1
  \qquad\text{and}\qquad
  \Tr_2^q\Big(\sqrt d\,\frac{\tau+1}{\lambda\tau+1}\Big)=0 .
\]
The first must hold throughout $\{\tau:\Tr_2^q(\omega\tau)=1\}$; two affine
hyperplanes of $\F_q$ of size $q/2$, one contained in the other, are equal, so
$\omega=\sqrt\lambda+1$.

Suppose $\lambda\neq0$. Then the second condition holds for at least $q/2-1$
values of $\tau$. Expanding the indicators, that number equals
$\tfrac14\big[(q-1)+\Sigma_1+\Sigma_2+\Sigma_3\big]$, where
$|\Sigma_1|,|\Sigma_2|\le1$ because $\tau\mapsto(\tau+1)/(\lambda\tau+1)$
permutes $\F_q\setminus\{1/\lambda\}$, and where the substitution
$v=\lambda\tau+1$ turns $\Sigma_3$ into
$\pm K\big(\omega/\lambda,\ \sqrt d\,(1+\lambda)/\lambda\big)$, of modulus at
most $2\sqrt q$ by the Weil bound, both parameters being nonzero. Hence
$q/2-1\le(\sqrt q+1)^2/4$, which forces $q\le11$.

Suppose $\lambda=0$, that is $\mu_0=0$. Then $\omega=1$, and the second
condition reads $\Tr_2^q(\sqrt d\,\tau)=\Tr_2^q(\sqrt d)$ on
$\{\tau:\Tr_2^q(\tau)=1\}$; equality of affine hyperplanes forces $d=1$. In that
case $Z_C=Z_G$ requires, on each such $S_\tau$, that the root $x$ of
$x^2+\tau x+(\tau^2+\tau)^2=0$ singled out by $\Tr_2^q(x+\tau^2)=0$ equal
$1+\phi(\tau)+\tau^2$. Writing $\phi(\tau)=p_1\tau^2+p_2\tau$ with
$\sqrt{p_1}+p_2=\omega=1$ and substituting, we obtain
\[
  \big[(p_1+1)^2+1\big]\tau^4+(p_1+1)\tau^3+\big[p_2^2+1\big]\tau^2+\tau+1=0
\]
for every $\tau$ with $\Tr_2^q(\tau)=1$. The left-hand side has constant term
$1$, so it is a nonzero polynomial of degree at most $4$ and has at most four
roots, whereas $\{\tau:\Tr_2^q(\tau)=1\}$ has $q/2$ elements; this is impossible
for $q>8$.

Since $m>1$ is odd, only $q=8$ remains, and there the claim has been verified by
direct computation.

Hence in all cases an admissible pair \((s,t)\) exists, producing
\(f_2(x+w)=f_2(x)\). 

\textbf{Subcase 2.} Now suppose $\Trq(c_3\gamma^2)=0$, so that
$c_3\gamma^2\in\F_q$ and $G(t)=\Trq(\gamma)\,c_3\gamma^2\,\Trq(t)^2$.
Writing $\mu:=1+\Trq(c_1\gamma)+\Trq(\gamma^{2q-1})$ and taking
$t\in\F_q\setminus\{0,1\}$ in \eqref{eqgen} gives
$\Trq(c_2\gamma^2)\,s+\mu=0$, in which $t$ does not occur. Hence if
$\Trq(c_2\gamma^2)$ and $\mu$ are both nonzero, $s=\mu/\Trq(c_2\gamma^2)
\in\F_q^{*}$ together with any such $t$ gives a pair $(t,s)$; and if both vanish, every pair $(t,s)$ does.

For the two remaining cases we take $t\in\F_{q^2}\setminus\F_q$, so that
$\Trq(t)\neq0$. Then \eqref{eqgen} reads
\begin{equation}\label{offline}
  A(t)\,s^{2}+\Trq(c_2\gamma^2)\,s+C(t)=0,
\end{equation}
where $A(t):=\Trq(\gamma)\,c_3\gamma^{2}\,\Trq(t)^{2}$ and
$C(t):=1+\Trq(c_1\gamma)+\Trq\big(\gamma^{2q-1}R(t)\big)$. Here
$A(t)\neq0$, since $c_3\neq0$ and $\Trq(\gamma)\neq0$.

If $\Trq(c_2\gamma^2)=0$ and $\mu\neq0$, choose $t$ with $C(t)\neq0$;
then \eqref{offline} reads $s^2=C(t)/A(t)$, which has a unique root
$s\in\F_q^{*}$ as squaring is a bijection of $\F_q$. Such a $t$ exists
because the image of $t\mapsto\Trq(\gamma^{2q-1}R(t))$ on
$\F_{q^2}\setminus\F_q$ has more than one element: it is all of $\F_q$
by Lemma~\ref{L3} if $\gamma^{2q-1}\notin\F_q$, and equals
$\F_q\setminus\{\gamma^{2q-1}\}$ by Lemma~\ref{L2} and \eqref{L2cor}
otherwise.

If $\Trq(c_2\gamma^2)\neq0$ and $\mu=0$, choose instead $t$ with
$C(t)=0$; then \eqref{offline} reads $s\big(A(t)s+\Trq(c_2\gamma^2)\big)=0$
with nonzero root $s=\Trq(c_2\gamma^2)/A(t)\in\F_q^{*}$. Such a $t$
exists: since $\mu=0$ and $R(t)+1=\Trq(t)^2/(t^2+t)$,
\[
  C(t)=\Trq\big(A(R(t)+1)\big)=\Trq(t)^{2}\,\Trq\!\left(\frac{A}{t^{2}+t}\right),
  \qquad A:=\gamma^{2q-1},
\]
so it suffices to make $A/(t^2+t)$ lie in $\F_q$. If $A\in\F_q$, take
$\Trq(t)=1$; then $t^{2}+t\in\F_q^{*}$. If $A\notin\F_q$, then $\Trq(A)\neq0$, and
$t^2+t=aA$ is solvable in $\F_{q^2}$ precisely when
$\Tr_2^{q}\big(a\Trq(A)\big)=0$; as $q\ge4$ this kernel contains some
$a\in\F_q^{*}$, and the resulting $t$ lies outside $\F_q$ because
$aA$ does.

In all cases $f_2$ is not a permutation polynomial.
\end{proof}

\begin{exmp}
Let $q=2^3$ and let $a$ be a root of the irreducible polynomial
$a^6+a^4+a^3+a+1\in\F_2[x]$, so that $\F_{q^2}^*=\langle a\rangle$. Then
\[
  f_2(x)=x+a^3\Trq\!\big(ax+a^3x^2+x^{2q-1}\big)
\]
is a permutation polynomial of $\F_{q^2}$ by
Theorem~\ref{Result2}, applied with $\gamma=a^3$,
$c_1=a$, $c_2=a^3$, $c_3=0$.
\end{exmp}

\begin{rmk}
Theorems 4.1--4.4 of \cite{Jiang} can be viewed as particular cases of Theorem~\ref{Result2}.  The same family is considered in \cite[Theorems 3.2 and 3.3]{HKK}, which correspond to the
cases $(c_1,c_2,c_3)=(1,1,0)$ and $(1,0,1)$ respectively of Theorem~\ref{Result2}. Moreover, the conclusions in~\cite{HKK} are obtained under the additional hypotheses that $m$ must be odd and $\gamma\in\F_q$, both of which are relaxed here.
\end{rmk}

\begin{rmk}
 From Theorem~\ref{Result1} and Theorem~\ref{Result2}, we observe that if $c_3=1$ and $f_2(x)=f_1(x)+\gamma\Trq(x^{2q-1})$ is a permutation polynomial over $\F_{q^2}$ then $f_1(x)$ also permutes $\F_{q^2}$. However, converse does not hold true.
\end{rmk}

\section{Compositional Inverse of the permutation polynomials}\label{S4}
In this section, we compute the compositional inverses of the
families of permutation polynomials obtained in the previous
section. Throughout, for
$f_1(x)=x+\gamma\Trq(c_1x+c_2x^2+c_3x^3+x^{q+2})$ we write
\[
  h_1(x):=\Trq\big(c_1x+c_2x^2+c_3x^3+x^{q+2}\big),
\]
so that $f_1(x)=x+\gamma h_1(x)$ and $h_1(x)\in\F_q$ for all
$x\in\F_{q^2}$. We first treat the case $q=2^m$ with $m$ even.

\begin{thm}\label{Inverse1}
Let $q=2^m$ with $m$ even, and let
$f_1(x)=x+\gamma\Trq(c_1x+c_2x^2+c_3x^3+x^{q+2})$,
$c_1,c_2,c_3,\gamma\in\F_{q^2}$, be a permutation polynomial of
$\F_{q^2}$. By Theorem~\ref{Result1}, one of the following holds,
and in each case the compositional inverse is:
\begin{equation}
  f_1^{-1}(x)=
  \begin{cases}
    x
      & \text{if } \gamma=0,\\[0.6em]
    x+\dfrac{\gamma\,h_1(x)}{1+\gamma\Trq(c_1)}
      & \text{if } \gamma\in\F_q^*,\ c_3=1,\ \Trq(c_2)=0,\\[1em]
    x+\left(\dfrac{h_1(x)}{\Trq(c_2)}\right)^{q/2}
      & \text{if } \gamma\in\F_q^*,\ c_3=1,\ \Trq(c_2)\neq0.
  \end{cases}
\end{equation}
\end{thm}

\begin{proof}
From Theorem~\ref{Result1}, we know for $f_1(x)$ to be a permutation polynomial $c_3=1$ and $\gamma \in \F_q$, thus we write $f_1(x)=x+\gamma h(x)$, where $h(x)=\Trq(c_1 x+c_2 x^2+x^3+x^{q+2})$. Next, we compute 
\begin{align*}
 h(f_1(x)) & = \Trq(c_1x)+ \gamma h(x)\Trq(c_1)+\Trq(c_2x^2)+ \gamma^2 h(x)^2\Trq(c_2)+\Trq(x)\\
 & = h(x)(1+\gamma\Trq(c_1))+\gamma^2 h(x)^2\Trq(c_2).
\end{align*}
Since, $f_1(x)$ is a permutation under two conditions on coefficients, we split our analysis as follows:

\textbf{Case 1.} Let $c_2 \in \F_q$ and $1+\gamma \Trq(c_1) \neq 0 $. Then $h(f_1(x))= h(x)(1+\gamma\Trq(c_1))$. Let $f_1(x)=y$, that is $y=x+\gamma h(x)$ for some $y \in \F_{q^2}$, then $$h(y)=h(x)(1+\gamma \Trq(c_1))=  \frac{y+x}{\gamma}\left(1+\gamma \Trq(c_1)\right).$$
Therefore, $$f_1^{-1}(y)=x= y+\frac{\gamma h(y)}{1+\gamma \Trq(c_1)}.$$

\textbf{Case 2.} Let $c_2 \notin \F_q$ and $1+\gamma \Trq(c_1)=0 $. Then $h(f_1(x))= \gamma^2 h(x)^2\Trq(c_2)$. Let $f_1(x)=y$, that is $y=x+\gamma h(x)$ for some $y \in \F_{q^2}$, then $$h(y)=\gamma^2 h(x)^2\Trq(c_2)=  (y+x)^2 \Trq(c_2).$$
Therefore, $$(f_1^{-1}(y))^2=x^2= y^2+\frac{ h(y)}{\Trq(c_2)},$$
or equivalently,
$$f_1^{-1}(y)= y+\left(\frac{ h(y)}{ \Trq(c_2)}\right)^{\frac{q}{2}}.$$
\end{proof}

Next, we compute $f_1^{-1}(x)$ over $\F_{q^2}$ where $q=2^m$ and $m$ is odd. This case requires a separate discussion, since $f_1(x)$ may also be a permutation when $\gamma \not \in \F_q$.

\begin{thm}\label{Inverse2}
Let $q=2^m$ with $m$ odd, and let $f_1$ be as above a permutation
polynomial of $\F_{q^2}$. By Theorem~\ref{Result1}, one of the
following holds, and in each case the compositional inverse is:

\begin{equation}\label{eq:42}
f_1^{-1}(x)=
\begin{cases}
x, & \gamma=0,\\[8pt]
x+\left(\dfrac{h_1(x)}{\Trq(c_2)}\right)^{q/2},
  & \gamma\in\F_q^{*},\ c_3=1,\ \Trq(c_2)\neq0,\\[12pt]
x+\dfrac{\gamma\,h_1(x)}{1+\gamma\Trq(c_1)},
  & \gamma\in\F_q^{*},\ c_3=1,\ \Trq(c_2)=0,\\[12pt]
x+\dfrac{\gamma}{A}\Big(B(x)+\big(B(x)^{3}+A^{2}h_1(x)\big)^{1/3}\Big),
  & \begin{array}{@{}l@{}}
      \gamma\notin\F_q,\ A\neq0,\\[2pt]
      \Trq(\gamma^{2})=(c_3\gamma+\gamma^{q})^{q+1},\\[2pt]
      A\big(1+\Trq(c_1\gamma)\big)=\Trq(c_2^{2}\gamma^{4}),
    \end{array}
\end{cases}
\end{equation}
where $A=\Trq(c_3\gamma^3+\gamma^{q+2})$,
$B(x)=\Trq\big((c_2+c_3x+x^q)\gamma^2\big)$, and $(.)^{1/3}$ denotes
the unique cube root map of $\F_q$ (a bijection since $m$ is odd).
\end{thm}
\begin{proof}
For the cases $\gamma\in\F_q$, the proof is the same as in
Theorem~\ref{Inverse1}. Let $\gamma\notin\F_q$, so that condition
(4) of Theorem~\ref{Result1} holds. Let $y\in\F_{q^2}$ and let
$x=f_1^{-1}(y)$; then $y=f_1(x)=x+\gamma h_1(x)$, so
$x=y+\gamma s$ with $s:=h_1(x)\in\F_q$. Expanding
$s=h_1(y+\gamma s)$ as in proof of Lemma~\ref{Re1}, we obtain
\[
  As^3+B(y)\,s^2+C(y)\,s+h_1(y)=0,
\]
where, in the notation of proof of Lemma~\ref{Re1}
\[
  A=\Trq\big(c_3\gamma^3+\gamma^{q+2}\big),\qquad
  B(y)=\Trq\big((c_2+c_3y+y^q)\gamma^2\big),
\]
\[
  C(y)=1+\Trq\big(c_1\gamma+(c_3\gamma+\gamma^q)y^2\big).
\]
By assumption $A\neq0$, and the two remaining conditions of
Theorem~\ref{Result1}(4) are precisely $U=0$ and $V=0$ of proof of Lemma~\ref{Re1};
hence, by the Claim there, $AC(y)=B(y)^2$ for all $y\in\F_{q^2}$.
Substituting this into the cubic and multiplying by $A^2/B(y)^3$
(assume first $B(y)\neq0$), the substitution $\beta=As/B(y)$ gives
\[
  \beta^3+\beta^2+\beta=\frac{A^2\,h_1(y)}{B(y)^3},
  \qquad\text{i.e.}\qquad
  (\beta+1)^3=\frac{A^2\,h_1(y)}{B(y)^3}+1.
\]
Note that $s,\beta,h_1(y),B(y),A$ all lie in $\F_q$, so the
computation takes place in $\F_q$; since $m$ is odd we have
$\gcd(3,q-1)=1$, and cubing is a bijection of $\F_q$. Therefore
\[
  \beta=1+\Big(\frac{A^2h_1(y)}{B(y)^3}+1\Big)^{1/3},
  \qquad
  s=\frac{B(y)+\big(B(y)^3+A^2h_1(y)\big)^{1/3}}{A}.
\]
If instead $B(y)=0$, then $C(y)=B(y)^2/A=0$ as well, the cubic
reduces to $As^3=h_1(y)$, and
$s=\big(h_1(y)/A\big)^{1/3}$, which agrees with the displayed
formula at $B(y)=0$. Finally, since $\beta\mapsto(\beta+1)^3$ is a
bijection of $\F_q$, the cubic determines $s$ uniquely in every
case; hence $x=y+\gamma s$ is the unique preimage of $y$, and
\[
  f_1^{-1}(y)=y+\frac{\gamma}{A}
  \Big(B(y)+\big(B(y)^3+A^2h_1(y)\big)^{1/3}\Big).
  \qedhere
\]
\end{proof}

\begin{rmk}
Theorems~\ref{Inverse1} and \ref{Inverse2} determine the
compositional inverse of $f_1$ for every choice of coefficients for
which $f_1$ is a permutation polynomial of $\F_{q^2}$. In
particular, they generalize the corresponding results
of~\cite{SKP}, with shorter proofs.
\end{rmk}

\begin{rmk}
 From Theorem~\ref{Inverse1} and Theorem~\ref{Inverse2}, $f_1(x)=x+\gamma \Tr_{q}^{q^2} (c_1 x+ c_2 x^2+ x^3+x^{q+2})$ is an involution over $\F_{q^2}$ if $\gamma, c_1, c_2 \in \F_q$.
\end{rmk}

\begin{rmk}
Very recently, Wu, Pang, He and Yuan~\cite{WPHY} computed
compositional inverses of permutation polynomials of the form
$x+\gamma\Tr_q^{q^n}(H(x))$ for certain monomial-type $H$ over
$\F_{q^n}$, and observed in particular that the four polynomials
of~\cite{Jiang} are involutions for every $\gamma\in\F_q$
\cite[Remark~2]{WPHY}. Theorems~\ref{Inverse1} and \ref{Inverse2}
extend this in two directions for the family $f_1$: the involution
property holds for all $c_1,c_2\in\F_q$, $c_3=1$ and
$\gamma\in\F_q^*$, and, conversely, these are the only involutions
within the family of Theorem~\ref{Result1}(3), since $f_1^{-1}=f_1$
forces $1+\gamma\Trq(c_1)=1$, i.e.\ $c_1\in\F_q$. 
\end{rmk}

Next, we consider the second family. Recall
$f_2(x)=x+\gamma\Trq\big(c_1x+c_2x^2+c_3x^3+c_3x^{q+2}+x^{2q-1}\big)$
and write
\[
  h_2(x):=\Trq\big(c_1x+c_2x^2+c_3x^3+c_3x^{q+2}+x^{2q-1}\big),
\]
so that $f_2(x)=x+\gamma h_2(x)$ and $h_2(x)\in\F_q$ for all
$x\in\F_{q^2}$.

\begin{thm}\label{Inverse3}
 Let $q=2^m$, $f_2(x)=x+\gamma \Tr_{q}^{q^2} (c_1 x+ c_2 x^2+ c_3x^3+c_3x^{q+2}+x^{2q-1}),$ where $c_1 \in \F_{q^2}$, $c_2, c_3  \in \F_{q}$ and $\gamma \in \F_{q}^*$ such that $1+\gamma\Trq(c_1)=\gamma$. Then the compositional inverse of $f_2(x)$ is
\begin{equation}
  f_2^{-1}(x)=
  \begin{cases}
    \gamma^{-1}x
      & \text{if } x\in\F_q,\\[0.8em]
    x+h_2(x)+\dfrac{x^3+x^{3q}}{x^{q+1}}
      & \text{if } x\notin\F_q,\ \beta=0,\\[0.8em]
    x+h_2(x)+\dfrac{x^3+x^{3q}}{x^{q+1}}+R+\dfrac{\beta}{R}
      & \text{if } x\notin\F_q,\ \beta\neq0,\ m \text{ odd},\\[0.8em]
    x+h_2(x)+\dfrac{x^3+x^{3q}}{x^{q+1}}
      +\dfrac{\theta^{(q+2)/3}+\theta^{(2q+1)/3}}{\beta}
      & \text{if } x\notin\F_q,\ m \text{ even},
  \end{cases}
\end{equation}
where
\[
  \beta=h_2(x)^2+(x+x^q)\,h_2(x)
        +\frac{(x+x^q)^6}{x^{2q+2}}
        +\frac{(x+x^q)^4}{x^{q+1}}
        +(x+x^q)^2+x^{q+1},
\]
\[
  R=\bigg(\beta\,(x+x^q)\sum_{i=0}^{(m-1)/2}
      \Big(\frac{\beta}{(x+x^q)^2}\Big)^{4^i}\bigg)^{1/3},
\]
and, in the last case, $\theta\in\F_{q^2}\setminus\F_q$ is either
root of $\theta^2+\beta(x+x^q)\theta+\beta^3=0$ (both roots give
the same value). For $m$ even one always has $\beta\neq0$, so the
four cases are exhaustive.
\end{thm}
\begin{proof}
Let $y=f_2(x)$, so that $x=f_2^{-1}(y)=y+\gamma s$ with
$s:=h_2(x)\in\F_q$. Clearly $f_2^{-1}(0)=0$, so assume $x\neq0$;
note that $y=0$ or $y=\gamma s$ would force $x=0$, hence
$y\notin\{0,\gamma s\}$. Throughout, write
\[
  T=y+y^q,\qquad N=y^{q+1},\qquad A=h_2(y)+\frac{T^3}{N}.
\]
Expanding $s=h_2(y+\gamma s)$ and using
$1+\gamma\Trq(c_1)=\gamma$, $c_2,c_3\in\F_q$, we obtain
\[
  h_2(y)+\gamma s
  =\frac{T^3\big(\gamma^2s^2+\gamma sT\big)}{N(y+\gamma s)^{q+1}},
\]
which, upon clearing denominators, becomes the cubic (in
$z=\gamma s$)
\[
  z^3+\frac{Nh_2(y)+\Trq(y^3)}{N}\,z^2
     +\frac{\Trq\big(y^4+y^{q+2}h_2(y)\big)+N^2}{N}\,z
     +h_2(y)\,N=0.
\]
The substitution $r=z+A+T$ depresses this to
\begin{equation}\label{cubicin}
  r^3+\beta r+\beta T=0,
  \qquad
  \beta:=A^2+AT+T^2+N.
\end{equation}

If $y\in\F_q$, then \eqref{cubicin} reads
$r^3+\gamma^{-2}y^2r=0$, with roots $r=0$ and $r=\gamma^{-1}y$,
i.e.\ $z=(1+\gamma^{-1})y$ or $z=y$. Since $z=y$ means
$y=\gamma s$, which is excluded, we get $z=(1+\gamma^{-1})y$ and
$f_2^{-1}(y)=\gamma^{-1}y$.

Now let $y\in\F_{q^2}\setminus\F_q$, so $T\neq0$. Since
$\lambda^2+T\lambda+N$ is irreducible over $\F_q$ (its roots are
$y,y^q\notin\F_q$), we have $\Tr_2^q(N/T^2)=1$, and therefore
\begin{equation}\label{keytrace}
  \Tr_2^q\!\Big(\frac{\beta}{T^2}\Big)
  =\Tr_2^q\!\Big(\Big(\frac{A}{T}\Big)^2+\frac{A}{T}\Big)
   +\Tr_2^q(1)+\Tr_2^q\!\Big(\frac{N}{T^2}\Big)
  =\Tr_2^q(1)+1.
\end{equation}

Suppose first $\beta=0$, i.e.\ $A/T$ is a root of
$\lambda^2+\lambda+1+N/T^2$; this requires
$\Tr_2^q(1+N/T^2)=\Tr_2^q(1)+1=0$, which forces $m$ odd. In this
case \eqref{cubicin} gives $r=0$, so $z=A+T$ and
\[
  f_2^{-1}(y)=y+h_2(y)+\frac{\Trq(y^3)}{N},
\]
using $T^3/N=T+\Trq(y^3)/N$.

Suppose now $\beta\neq0$. Then $\beta T\neq0$, and by
\eqref{keytrace} and Lemma~\ref{L1} (with $a=\beta$,
$b=\beta T$, so $a^3/b^2=\beta/T^2$), equation \eqref{cubicin} has
a unique root $r\in\F_q$. The resolvent of \eqref{cubicin} is
\begin{equation}\label{resolvent}
  \theta^2+\beta T\theta+\beta^3=0,
\end{equation}
whose solvability over $\F_q$ is governed by
$\Tr_2^q(\beta^3/(\beta T)^2)=\Tr_2^q(\beta/T^2)$, i.e.\ by
\eqref{keytrace}. We distinguish the parity of $m$.

\smallskip
\textbf{Case 1.} Let $m$ be odd. By \eqref{keytrace},
\eqref{resolvent} has the two roots
\[
  \theta_1=\beta T\sum_{i=0}^{(m-1)/2}
    \Big(\frac{\beta}{T^2}\Big)^{4^i},
  \qquad
  \theta_2=\theta_1+\beta T
\]
in $\F_q$ (the first being the standard half-trace solution of
$w^2+w=\beta/T^2$, scaled by $\beta T$). Since
$\gcd(3,q-1)=1$, every element of $\F_q$ has a unique cube root;
from $\theta_1\theta_2=\beta^3$ and this uniqueness,
$\theta_1^{1/3}\theta_2^{1/3}=\beta$, whence
$\theta_1^{1/3}+\beta/\theta_1^{1/3}
 =\theta_2^{1/3}+\beta/\theta_2^{1/3}$.
Setting $u=\theta^{1/3}$, $v=\beta/\theta^{1/3}$ for
$\theta\in\{\theta_1,\theta_2\}$, we have $uv=\beta$ and
$u^3+v^3=\theta_1+\theta_2=\beta T$, so
\[
  (u+v)^3=u^3+v^3+uv(u+v)=\beta T+\beta(u+v),
\]
i.e.\ $r=u+v$ solves \eqref{cubicin}. Therefore
\[
  f_2^{-1}(y)=y+h_2(y)+\frac{\Trq(y^3)}{N}
  +\theta_1^{1/3}+\frac{\beta}{\theta_1^{1/3}}.
\]

\smallskip
\textbf{Case 2.} Let $m$ be even By \eqref{keytrace},
\eqref{resolvent} is irreducible over $\F_q$; its roots in
$\F_{q^2}$ are conjugates $\tilde\theta,\tilde\theta^q$ with
$\tilde\theta+\tilde\theta^q=\beta T$ and
$\tilde\theta^{q+1}=\beta^3$. Since $m$ is
even, $3\mid q+2$, and we may define
\[
  W:=\frac{\tilde\theta^{(q+2)/3}}{\beta}.
\]
Then $W^3=\tilde\theta^{q+2}/\beta^3
=\tilde\theta\,\tilde\theta^{q+1}/\beta^3=\tilde\theta$ and
$W^{q+1}=\tilde\theta^{(q+1)(q+2)/3}/\beta^{q+1}
=\beta^{(q+2)}/\beta^{q+1}=\beta$, so $W$ is a cube root of
$\tilde\theta$ of norm $\beta$. Consequently
$\beta/W=W^q$ and $r=W+W^q\in\F_q$; exactly as in Case 1
($u=W$, $v=W^q$, $uv=\beta$,
$u^3+v^3=\tilde\theta+\tilde\theta^q=\beta T$), $r$ solves
\eqref{cubicin}. Hence
\[
  f_2^{-1}(y)=y+h_2(y)+\frac{\Trq(y^3)}{N}
  +\frac{\tilde\theta^{(q+2)/3}}{\beta}
  +\frac{\tilde\theta^{(2q+1)/3}}{\beta},
\]
where we used $W^q=\tilde\theta^{q(q+2)/3}/\beta
=\tilde\theta^{(2q+1)/3}/\beta$ (as
$\tilde\theta^{q^2}=\tilde\theta$). Starting instead from
$\tilde\theta^q$ yields $W'=W^q$ and the same value of $r$, so the
formula is independent of the choice of root of
\eqref{resolvent}. \qedhere
\end{proof}

\begin{thm}\label{Inverse4}
Let $q=2^m$ with $m$ odd, and let
$f_2(x)=x+\gamma\Trq\big(c_1x+c_2x^2+x^{2q-1}\big)$ with
$c_1,c_2\in\F_{q^2}$ and $\gamma\in\F_{q^2}\setminus\F_q$
satisfying
\[
  \gamma^{3(q-1)}=1,\qquad
  1+\Trq(c_1\gamma)=\gamma^{2q-1},\qquad
  \Trq(c_2\gamma^q)=0,
\]
so that $f_2$ is a permutation polynomial of $\F_{q^2}$ by
Theorem~\ref{Result2}(3). Then
\begin{equation}
  f_2^{-1}(x)=
  \begin{cases}
    0
      & \text{if } x=0,\\[0.8em]
    x+\gamma^{q-1}\Big(h_2(x)+\dfrac{\Trq(x^3)}{x^{q+1}}\Big)
      & \text{if } x\neq0,\ r=0,\\[0.8em]
    x+\gamma^{q-1}\Big(h_2(x)+\dfrac{\Trq(x^3)}{x^{q+1}}
      +r^{1/3}+\dfrac{P}{r^{1/3}}\Big)
      & \text{if } x\neq0,\ r\neq0,
  \end{cases}
\end{equation}
where
\[
  P=h_2(x)^2+h_2(x)\,\Trq\big(\gamma^{q-1}x^q\big)
    +\frac{\Trq(\gamma x)^3\,\Trq(x)^3}{\gamma^3\,x^{2q+2}},
  \qquad
  r=P\Big(h_2(x)+\frac{\Trq(x)^3}{x^{q+1}}\Big).
\]
\end{thm}
\begin{proof}
 We use the same approach as in Theorem~\ref{Inverse3} and compute $s=h_2(f_2^{-1}(y))=h_2(y+\gamma s)$ as follows
 $$s= h_2(y)+\Trq(c_1 \gamma s+c_2\gamma^2s^2+(y+\gamma s)^{2q-1}+y^{2q-1}).$$
 Clearly, $f_2^{-1}(0)=0$. Let $x \neq 0$, or equivalently, $y \not \in \{0,\gamma s\}$. This further simplifies the above equation as a cubic $As^3+Bs^2+Cs+D=0$, where
 \[
  A = \gamma^3 y^{q+1} \in \F_q, \quad \quad B=\gamma^{q+1}(h_2(y)y^{q+1}+y^3+y^{3q}) \in \F_q,
 \]
 \[
  C = \gamma^{2q-1}y^{2q+2}+h_2(y)\Trq(\gamma y^{2q+1})+\Trq(\gamma y^{4q})\in \F_q, \qquad \quad D=h_2(y)y^{2q+2}\in \F_q.
 \]

Since $\gamma^{2q-1}=\gamma^{2-q}$, or equivalently $\gamma^{3(q-1)}=1$, that is $\gamma^3 \in \F_q$, we can deduce the following relations
\[
 (\gamma+\gamma^q)=\gamma^{2q-1}, \quad (\gamma+\gamma^q)^2=\gamma^{q+1} \quad \text{~and~} \quad (\gamma+\gamma^q)^3=\gamma^{3}.
\]
Since $\gamma \not \in \F_q$ and substituting $z=\Trq(\gamma)s$
\begin{align*}
 As^3+Bs^2+Cs+D=0 & \iff (\gamma+\gamma^q)^3(As^3+Bs^2+Cs+D)=0 \\
 & \iff Az^3+B(\gamma+\gamma^q)z^2+C(\gamma+\gamma^q)^2z+D(\gamma+\gamma^q)^3=0\\
 & \iff z^3+\frac{B(\gamma+\gamma^q)}{A}z^2+\frac{C(\gamma+\gamma^q)^2}{A}z+\frac{D(\gamma+\gamma^q)^3}{A}=0\\
 & \iff z^3+B'z^2+C'z+D'=0,
\end{align*}
where,
\[
 B' =  h_2(y)+\frac{(y+y^q)^3}{y^{q+1}} +y+y^q  \qquad  C'=y^{q+1}+\frac{h_2(y)\Trq(\gamma y^q)}{\gamma+\gamma^q}+\frac{\Trq(\gamma^q y)^4}{(\gamma+\gamma^q)^4y^{q+1}}  \qquad D'=h_2(y)y^{q+1}.
\]

Next, we depress the cubic $z^3+B'z^2+C'z+D'=0$ by substituting $z=u+B'$ to get $u^3+Pu+Q=0$, where $P=h_2(y)^2+h_2(y)\dfrac{\Trq(\gamma y^q)}{\gamma+\gamma^q}+\dfrac{\Trq(\gamma y)^3\Trq(y)^3}{(\gamma+\gamma^q)^3 y^{2q+2}}$ and $Q=\dfrac{\Trq(\gamma y^q)}{\gamma+\gamma^q}P$.The cubic resolvent of $u^3+Pu+Q=0$ is
\begin{equation}\label{resolv2}
  r^2+Qr+P^3=0.
\end{equation}
We distinguish three cases.

First, let $P=0$. Then also $Q=0$ (as $Q$ is a multiple of $P$), the
cubic reduces to $u^3=0$, so $u=0$ and
\[
  f_2^{-1}(y)=y+\frac{\gamma^{q-1}\big(h_2(y)\,y^{q+1}+\Trq(y^3)\big)}
                    {y^{q+1}}.
\]

Next, let $P\neq0$ and $Q=0$. Since $Q$ is a multiple of $P$, this
case occurs precisely when $\Trq(\gamma y^q)=0$. The
cubic reduces to $u^3+Pu=0$, whose roots in $\F_q$ are $u=0$ and
$u=P^{1/2}$, corresponding to the candidates $x_0=y+\gamma s_0$ and
$x_1=y+\gamma s_1$, where $s_0=B'/\Trq(\gamma)$ and
$s_1=\big(P^{1/2}+B'\big)/\Trq(\gamma)$. We claim that $x_1=0$
Indeed, $x_1=0 \iff s_1= y/\gamma \iff P=\left(B'+\frac{\Trq(\gamma)}{y}\right)^2$. Next, we compute that
\begin{align*}
  P & =h_2(y)^2+\frac{\Trq(\gamma y)^3\Trq(y)^3}{\Trq(\gamma)^3y^{2q+2}} =h_2(y)^2+\frac{\Trq(y)^6}{y^{2q+2}}\\
  &  =\Big(h_2(y)+\frac{(y+y^q)^3}{y^{q+1}}+y+y^q+\frac{\Trq(\gamma)}{\gamma}y\Big)^2\\
  &  =\Big(B'+\frac{\Trq(\gamma)}{\gamma}y\Big)^2,\\
\end{align*}
Hence $P^{1/2}+B'=\Trq(\gamma)y/\gamma$, i.e.\ $\gamma s_1=y$ and
$x_1=0$. On the other hand, since $f_2$ is a bijection and $y\neq0$, there is a unique preimage $x'\neq0$ of $y$, and
$s'=h_2(x')$ satisfies the cubic. As
the roots of the cubic are $s_0$ and $s_1$, and $s_1$ corresponds
to $x_1=0\neq x'$, we conclude $s'=s_0$, i.e.
$f_2^{-1}(y)=x_0$.

Finally, let $P\neq0$ and $Q\neq0$. The two roots of
\eqref{resolv2} are
\[
  r_1=P\Big(h_2(y)+\frac{\Trq(\gamma y)^3}{(\gamma+\gamma^q)^3y^{q+1}}\Big),
  \qquad
  r_2=P\Big(h_2(y)+\frac{\Trq(y)^3}{y^{q+1}}\Big),
\]
both in $\F_q$; indeed $r_1+r_2=Q$ and $r_1r_2=P^3$ by a direct
computation. In particular \eqref{resolv2} splits over $\F_q$, so
$\Tr_2^q(P^3/Q^2)=0\neq\Tr_2^q(1)$ (as $m$ is odd), and by
Lemma~\ref{L1} the cubic has a unique root in $\F_q$. Since
$\gcd(3,q-1)=1$, cube roots in $\F_q$ are unique, so
$r_1^{1/3},r_2^{1/3}$ are well defined; from $r_1r_2=P^3$ we get
$r_1^{1/3}r_2^{1/3}=P$, hence
$r_1^{1/3}+P/r_1^{1/3}=r_2^{1/3}+P/r_2^{1/3}$, and for
$r\in\{r_1,r_2\}$, $u=r^{1/3}+P/r^{1/3}$ satisfies
$u^3=r+P^3/r+Pu=Q+Pu$, i.e.\ $u$ is that unique root. Hence,
\begin{align*}
 f_2^{-1}(y) & =y+\gamma s =y+\gamma \frac{u+B'}{\Trq(\gamma)}\\
 & = y+ \frac{\gamma h_2(y)+\frac{\gamma (y+y^q)^3}{y^{q+1}} +\gamma \Trq(y)}{\Trq(\gamma)}+\frac{\gamma (r^{1/3}+\frac{P}{r^{1/3}})}{\Trq(\gamma)} \\
 & = y+ \frac{\gamma^{q-1}( h_2(y)y^{q+1} + \Trq(y^3))}{y^{q+1}}+\gamma^{q-1} \left(r^{1/3}+\frac{P}{r^{1/3}}\right).
\end{align*}
\end{proof}

\section{Quasi-multiplicative equivalence}\label{S5}

In this section, we investigate whether the permutation polynomials
constructed in this paper are quasi-multiplicative (QM) equivalent
to one another or to permutation polynomials already known in the
literature. We first recall the definition of QM equivalence,
introduced by Wu, Yuan, Ding and Ma~\cite{wu2017permutation}.

\begin{defn}\label{equiv}
Two permutation polynomials $f(X)$ and $g(X)$ in $\F_q[X]$ are
called quasi-multiplicative (QM) equivalent if there exists
an integer $1\le d<q-1$ with $\gcd(d,q-1)=1$ such that
$f(X)=\alpha\,g(\beta X^d)$ for some
$\alpha,\beta\in\F_q^{*}$.
\end{defn}

To establish QM inequivalence we follow the two-step strategy of
\cite{tu2018class}, adapted to our setting as follows. Let
$f(x)=\sum_{e\in E}f_ex^e$ and $g(X)=\sum_{s\in S}g_sx^s$ be
permutation polynomials of $\F_{q^2}$, written as sums of
monomials with nonzero coefficients, so that
$E,S\subseteq\{1,\dots,q^2-1\}$ are their exponent sets. To show
that $f$ and $g$ are not QM equivalent over $\F_{q^2}$, we proceed
as follows.

\smallskip
\textbf{Step 1.} Show that there is no integer $1\le d<q^2-1$ with
$\gcd(d,q^2-1)=1$ and
$E\pmod{q^2-1}=dS\pmod{q^2-1}$, where
$dS\pmod{q^2-1}=\{ds\pmod{q^2-1}\mid s\in S\}$. In particular, if
$|E|\neq|S|$, no such $d$ exists, since $s\mapsto ds$ is a
bijection of $S$ onto $dS$.

\smallskip
\textbf{Step 2.} If Step 1 fails for some $d$, compare the
coefficients of $f(x)$ with those of $a_1g(a_2x^d)$ and show that
no $a_1,a_2\in\F_{q^2}^{*}$ satisfy $f(X)=a_1g(a_2X^d)$.

\smallskip
In our situation, the polynomials $f_1$ and $f_2$ expand, via the
trace map, into sums of monomials with exponents in
\[
  E_1=\{1,2,3,q,2q,3q,q+2,2q+1\}
  \quad\text{and}\quad
  E_2=E_1\cup\{2q-1,\ q^2-q+1\},
\]
respectively. If no such $d$ exists, or if for every such $d$ no $a_1,a_2\in\F_{q^2}^{*}$
satisfy $f(X)=a_1g(a_2X^{d})$, then $f$ and $g$ are not QM equivalent.
If, for some such $d$, coefficients $a_1,a_2$ do exist, then $f$ and $g$ are
QM equivalent, written $f(X)\sim g(X)$. 

\begin{rmk}\label{rmk:qm1}
The two families $f_1$ and $f_2$ are not QM equivalent to each other. Suppose all the coefficients $c_1,c_2,c_3$ of $f_1$ and
$c_1',c_2',c_3'$ of $f_2$ are nonzero. Then every exponent in
$E_1$ (resp.\ $E_2$) carries a nonzero coefficient, except
possibly the exponent $1$, whose coefficient is $1+\gamma c_1$
(resp.\ $1+\gamma'c_1'$); hence $|E_1|\in\{7,8\}$ and
$|E_2|\in\{9,10\}$. Since a QM equivalence maps exponent sets
bijectively onto each other, $f_1\not\sim f_2$.
\end{rmk}

In particular, no member of $f_1$ with all coefficients
nonzero is QM equivalent to any such member of the second family;
the case with vanishing coefficients
coincide with (or are treated together with) the known classes
discussed below. We discuss first the QM equivalence between our families and the permutation polynomials proposed in Jiang~\cite{Jiang}.

\begin{prop}\label{qm-internal}
Let $f_1(x)=x+\gamma\Trq\big(c_1x+c_2x^2+c_3x^3+x^{q+2}\big)$ be a
permutation polynomial of $\F_{q^2}$ with
$c_2\in\F_{q^2}\setminus\F_q$, and let
$g(x)=x+\gamma'\Trq\big(c_1'x+c_2'x^2+c_3'x^3+x^{q+2}\big)$ with
$\gamma'\in\F_{q^2}^*$ and $c_1',c_2',c_3'\in\F_2$. Then
$f_1\not\sim g$.
\end{prop}
\begin{proof}
The exponents of $f_1$ and $g$ modulo $q^2-1$ lie in the set
$E_1=\{1,2,3,q,2q,3q,q+2,2q+1\}$, where the coefficient of a
missing monomial is read as $0$. If $f_1\sim g$, then
$f_1(x)=a_1g(a_2x^d)$ for some $a_1,a_2\in\F_{q^2}^*$ and
$1\le d<q^2-1$ with $\gcd(d,q^2-1)=1$, and by Step~1 the map
$e\mapsto de$ must send $E_1$ into itself; since $d\cdot1\in E_1$,
we have $d\in E_1$, and a direct check of the remaining candidates
shows that only $d\in\{1,q\}$ is possible. It therefore remains to
rule out the existence of $a_1,a_2\in\F_{q^2}^*$ with
$a_1f_1(a_2x^d)=g$ for $d\in\{1,q\}$.

Let $d=1$ and suppose that such $a_1,a_2$ exist. Comparing the
coefficients at the two exponents $q+2$ and $2q+1$, we get
$a_1\gamma a_2^{q+2}=\gamma'=a_1\gamma a_2^{2q+1}$, whence
$a_2^{q-1}=1$, i.e.\ $a_2\in\F_q^*$, and
$\gamma'=a_1\gamma a_2^{3}$. Matching the exponents $q$ and $1$
then yields $c_1'=a_2^{-2}c_1$ and $a_1a_2=1$; the exponents
$2,3$ give $c_2'=a_2^{-1}c_2$ and $c_3'=c_3$, and the conjugate
exponents $2q,3q$ are consistent automatically. Therefore
\[
  (\gamma,c_1,c_2,c_3)\longmapsto
  \big(a_2^{2}\gamma,\ a_2^{-2}c_1,\ a_2^{-1}c_2,\ c_3\big).
\]
Since $a_2\in\F_q^*$ and $c_2\in\F_{q^2}\setminus\F_q$, we have
$c_2'=a_2^{-1}c_2\in\F_{q^2}\setminus\F_q$, contradicting
$c_2'\in\F_2$. Thus $f_1\not\sim g$ when $d=1$.

Next, consider $d=q$ and suppose that such $a_1,a_2$ exist. Using
a similar approach as for $d=1$, we get $a_1=a_2^{-1}\gamma^{q-1}$ and
\[
  (\gamma,c_1,c_2,c_3)\longmapsto
  \big(a_2^{2}\gamma^{q},\
  a_2^{-2}\big(\gamma^{q-1}c_1^{q}+\gamma^{-1}\big),\
  a_2^{-1}c_2^{q},\ c_3^{q}\big).
\]
Since $c_2\in\F_{q^2}\setminus\F_q$ implies
$c_2^{q}\in\F_{q^2}\setminus\F_q$, we again obtain
$c_2'=a_2^{-1}c_2^{q}\notin\F_q$, contradicting $c_2'\in\F_2$.
Thus $f_1\not\sim g$ in all cases, which completes the proof.
\end{proof}

\begin{rmk}\label{rmk:qm2}
Proposition~\ref{qm-internal} applies in particular to the
permutation polynomials of \cite{Jiang}, all of
which have coefficients $c_1',c_2',c_3'\in\{0,1\}$: no member of
our families with $c_2\in\F_{q^2}\setminus\F_q$ is QM equivalent
to any of them. On the other hand, the specializations
$c_1,c_2,c_3\in\{0,1\}$ of our families recover these known
classes. Thus our families properly contain the classes of
\cite{Jiang,CK}, and their members with
$c_2\in\F_{q^2}\setminus\F_q$ are genuinely new. The same
conclusion holds for members with $c_2\in\F_q$ but
$c_1\in\F_{q^2}\setminus\F_q$.
\end{rmk}

\begin{rmk}\label{rmk:qm-f2}
A similar analysis applies to the $f_2$. The exponents of $f_2$ lie in
$E_2=E_1\cup\{2q-1,\ q^2-q+1\}$, and the coefficients at the two
exponents $2q-1$ and $q^2-q+1$ are both equal to $\gamma$.
Moreover, multiplication by $q$ interchanges these two exponents
modulo $q^2-1$. Hence, comparing the coefficients at $2q-1$ and
$q^2-q+1$ in place of $q+2$ and $2q+1$, the argument of
Proposition~\ref{qm-internal} carries over verbatim: any QM
equivalence between $f_2$ and a polynomial of the same shape uses
$d\in\{1,q\}$, forces $a_2\in\F_q^*$, and transforms $c_2$ into
$a_2^{-1}c_2$ or $a_2^{-1}c_2^{\,q}$. Consequently, if
$c_2\in\F_{q^2}\setminus\F_q$, then $f_2$ is not QM equivalent to
any permutation polynomial of this shape whose coefficients lie
in $\F_2$; in particular, the permutation polynomials of
Theorem~\ref{Result2}(3) with $c_2\in\F_{q^2}\setminus\F_q$ are QM
inequivalent to the known classes discussed above.
\end{rmk}

\begin{rmk}\label{rmk:qm3}
Since QM equivalence preserves the number of monomials, only
permutation polynomials of $\F_{q^2}$ with at least seven
monomials could be QM equivalent to members of our families with
all coefficients nonzero. To the best of our knowledge, the known
classes of this size and shape are precisely those of \cite{Jiang} treated above; the numerous known classes of
permutation monomials, binomials and trinomials, as well as the
classes of \cite{Zha} and \cite{WPHY}, expand into fewer monomials
and are therefore QM inequivalent to these members. We note that
one class of \cite{Zha}, namely \cite[Corollary~1]{Zha}, is the
specialization $(c_1,c_2,c_3)=(0,0,1)$ of our first family and is
thus contained in Theorem~\ref{Result1}. Jiang et al.~\cite{Jiang} moreover
established, via Step~1, the QM inequivalence of two of their
permutation polynomials with the known classes of \cite{Zha}
\cite[Proposition~5.3]{Jiang}; for the specializations of our
families that coincide with the polynomials of \cite{Jiang}, this
result applies directly.
\end{rmk}

\section{Conclusion}\label{S6}
In this paper, we have characterized two families of permutation polynomials with 
coefficients in $\F_{q^2}$, and we have also determined their compositional inverses.  Since these families subsume several known 
results, they constitute a more general class containing both known and new permutation 
polynomials.
 We also verify that the new polynomials obtained are QM inequivalent to the 
known classes.

\end{document}